\documentclass[reqno,a4paper,twoside,english]{amsart}

\usepackage{quoting}

\usepackage[cmyk,usenames,dvipsnames,svgnames]{xcolor}

\usepackage{amssymb,dsfont,mathrsfs,verbatim,bm,mathtools, float, longtable, tabu, booktabs,amsrefs}
\usepackage{babel,changes}
\usepackage{microtype}
\usepackage{enumitem}
\usepackage{graphicx}
\usepackage[figuresright]{rotating}

\usepackage{float}
\usepackage[section]{placeins}

\usepackage{tikz, tikz-cd}
\usetikzlibrary{matrix}
\usetikzlibrary{shapes}
\usepackage{bm}

\usepackage{lmodern}

\usepackage[section]{placeins}
\usepackage{cancel}

\usepackage[raggedright]{titlesec}
\titleformat{\section}
{\normalfont\Large\bfseries\center}{\thesection}{1em}{}
\titleformat{\subsection}
{\normalfont\large\bfseries}{\thesubsection}{1em}{}
\titleformat{\subsubsection}
{\normalfont\normalsize\bfseries}{\thesubsubsection}{1em}{}
\titleformat{\paragraph}[runin]
{\normalfont\normalsize\bfseries}{\theparagraph}{1em}{}
\titleformat{\subparagraph}[runin]
{\normalfont\normalsize\bfseries}{\thesubparagraph}{1em}{}
\titlespacing*{\section} {0pt}{3.5ex plus 1ex minus .2ex}{2.3ex plus .2ex}
\titlespacing*{\subsection} {0pt}{3.25ex plus 1ex minus .2ex}{1.5ex plus .2ex}
\titlespacing*{\subsubsection}{0pt}{3.25ex plus 1ex minus .2ex}{1.5ex plus .2ex}
\titlespacing*{\paragraph} {0pt}{3.25ex plus 1ex minus .2ex}{1em}
\titlespacing*{\subparagraph} {\parindent}{3.25ex plus 1ex minus .2ex}{1em}

\makeatletter
\renewcommand*\env@matrix[1][\arraystretch]{%
  \edef\arraystretch{#1}%
  \hskip -\arraycolsep
  \let\@ifnextchar\new@ifnextchar
  \array{*\c@MaxMatrixCols c}}
\makeatother

\usepackage{titletoc}
\contentsmargin{2em}
\dottedcontents{section}[2.7em]{}{2.8em}{1pc}
\dottedcontents{subsection}[3.7em]{}{3.4em}{1pc}
\dottedcontents{subsubsection}[7.6em]{}{3.2em}{1pc}
\dottedcontents{paragraph}[10.3em]{}{3.2em}{1pc}

\newcommand{\ra}[1]{\renewcommand{\arraystretch}{#1}}

\newcommand{\fracpadding}{}
\newcommand{\setfracpadding}[1][2pt]{%
  \sbox0{$\frac{1}{2}$}%
  \dimen0=\ht0 \advance\dimen0 #1\relax
  \dimen2=\dp0 \advance\dimen2 #1\relax
  \edef\fracpadding{\vrule width 0pt height \the\dimen0 depth \the\dimen2\relax}%
}

\newcommand{\ot}{\otimes}
\newcommand{\hs}{\hspace{-0.6pt}}

\newcommand{\hf}{\hspace{-0.9pt}}
\newcommand{\hg}{\hspace{-2.8pt}}
\newcommand{\hl}{\hspace{-1.6pt}}

\DeclareMathOperator{\LBE}{LBE}
\DeclareMathOperator{\RBE}{RBE}
\DeclareMathOperator{\LSQ}{LSQ}
\DeclareMathOperator{\RSQ}{RSQ}

\DeclareMathOperator{\ide}{Id}

\DeclareMathOperator{\chr}{char}

\allowdisplaybreaks[4]

\numberwithin{table}{section}
\numberwithin{equation}{section}
\numberwithin{figure}{section}
\theoremstyle{plain}
\newtheorem{theorem}{Theorem}[section]
\theoremstyle{plain}
\newtheorem{lemma}[theorem]{Lemma}
\newtheorem{definition}[theorem]{Definition}
\newtheorem{corollary}[theorem]{Corollary}
\theoremstyle{plain}
\newtheorem{proposition}[theorem]{Proposition}
\newtheorem{notation}[theorem]{Notation}

\theoremstyle{remark}
\newtheorem{remark}[theorem]{Remark}

\theoremstyle{plain}

\newtheorem{examples}[theorem]{Examples}

\newdimen\CdotAxis
\newcommand*{\CdotAux}[3]{%
  {%
    \settoheight\CdotAxis{$#2\vcenter{}$}%
    \sbox0{%
      \raisebox\CdotAxis{%
        \scalebox{#1}{%
          \raisebox{-\CdotAxis}{%
            $\mathsurround=0pt #2#3$%
          }%
        }%
      }%
    }%
    \dp0=0pt %
    \sbox2{$#2\bullet$}%
    \ifdim\ht2<\ht0 %
      \ht0=\ht2 %
    \fi
    \sbox2{$\mathsurround=0pt #2#3$}%
    \hbox to \wd2{\hss\usebox{0}\hss}%
  }%
}

\begin{document}

\title[Solutions of the braid equation with set-type square]{Solutions of the braid equation with set-type square}

\author{Jorge A. Guccione}
\address{Departamento de Matem\'atica\\ Facultad de Ciencias Exactas y Naturales-UBA, Pabell\'on~1-Ciudad Universitaria\\ Intendente Guiraldes 2160 (C1428EGA) Buenos Aires, Argentina.}
\address{Instituto de Investigaciones Matem\'aticas ``Luis A. Santal\'o"\\ Pabell\'on~1-Ciudad Universitaria\\ Intendente Guiraldes 2160 (C1428EGA) Buenos Aires, Argentina.}
\email{vander@dm.uba.ar}

\author{Juan J. Guccione}
\address{Departamento de Matem\'atica\\ Facultad de Ciencias Exactas y Naturales-UBA\\ Pabell\'on~1-Ciudad Universitaria\\ Intendente Guiraldes 2160 (C1428EGA) Buenos Aires, Argentina.}
\address{Instituto Argentino de Matem\'atica-CONICET\\ Saavedra 15 3er piso\\ (C1083ACA) Buenos Aires, Argentina.}
\email{jjgucci@dm.uba.ar}

\thanks{Jorge A. Guccione and Juan J. Guccione were supported by UBACyT 20020150100153BA (UBA) and PIP 11220110100800CO (CONICET)}

\author{Christian Valqui}
\address{Pontificia Universidad Cat\'olica del Per\'u, Secci\'on Matem\'aticas, PUCP,
Av. Universitaria 1801, San Miguel, Lima 32, Per\'u.}

\address{Instituto de Matem\'atica y Ciencias Afines (IMCA) Calle Los Bi\'ologos 245. Urb San C\'esar.
La Molina, Lima 12, Per\'u.}
\email{cvalqui@pucp.edu.pe}

\thanks{Christian Valqui was supported by PUCP-DGI- ID 453 - CAP 2017-1-0035.}


\begin{abstract}
For a family of height one orders $(X,\le)$ and each non-degenerate solution $r_0\colon X\times X\longrightarrow X\times X$ of the set-theoretic braid equation on $X$ satisfying suitable conditions, we obtain all the non-degenerate solutions of the braid equation on the incidence coalgebra of $(X,\le)$ that extend~$r_0$.
\end{abstract}

\maketitle

\setcounter{tocdepth}{1}
\tableofcontents

\section*{Introduction}
Let $V$ be a vector space over a field $K$ and let $r\colon V\ot_K V\longrightarrow V\ot_K V$ be a bijective linear operator. We say that $r$ satisfies the {\em braid equation} if
\begin{equation*}
r_{12} \circ r_{23}\circ r_{12} = r_{23} \circ r_{12}\circ r_{23},
\end{equation*}
where $r_{ij}$ denotes $r$ acting on the $i$-th and $j$-th coordinates. Since the eighties many solutions of the braid equation have been found, many of them being deformations of the flip. It is interesting to obtain solutions that are not of this type, and in \cite{Dr}, Drinfeld proposed to study the most simple of them, namely, the set-theoretic ones, i.e. pairs $(X,r_0)$, where $X$ is a set and
$$
r_0\colon X\times X\longrightarrow X\times X
$$
is an invertible map satisfying the braid equation. Each one of these solutions yields in an evident way a linear solution on the vector space with basis~$X$. From a structural point of view this approach was considered first by Etingof, Schedler and Soloviev \cite{ESS} and Gateva-Ivanova and Van den Bergh \cite{GIVdB} for involutive solutions, and then by Lu, Yan and Zhu \cite{LYZ} and Soloviev \cite{So} for non-degenerate not necessarily involutive solutions. In the last two decades the theory has developed rapidly, and now it is known that it has connections with bijective 1-cocycles, Bierbach groups and groups of I-type, involutive Yang-Baxter groups, Garside structures, biracks, cyclic sets, braces, Hopf algebras, matched pairs, left symmetric algebras, etcetera (see, for instance \cite{AGV}, \cite{CJO}, \cite{CJO2}, \cite{CJR}, \cite{De}, \cite{GI}, \cite{Ru}, \cite{Ta}).

\smallskip

Suppose now that $(X,\le)$ is a locally finite poset and consider its incidence coalgebra~$D$. We identify each $a\in X$ with the  pair $(a,a)$ in $D$. In~\cite{GGV1} the following problem was posed:
\begin{quote}
Let $r_0\colon X\times X\longrightarrow X\times X$ be a non-degenerated solution of the set theoretical braid equation. Find necessary and sufficient conditions in order that $r_0$ is the restriction of a non-degenerate coalgebra automorphism $r$ of $D\ot D$, which is a solution of the braid equation, and then find all such extensions.
\end{quote}
Let $r\colon D\ot D \longrightarrow D\ot D$ be a linear map. For $a\le b$ and $c\le d$ write
\begin{equation}\label{def de los lambdas}
r((a,b)\ot (c,d))= \sum_{e\le f}\sum_{g\le h} \lambda_{a|b|c|d}^{e|f|g|h} (e,f)\ot (g,h),
\end{equation}
with $\lambda_{a|b|c|d}^{e|f|g|h}\in K$. Assume that $r$ is a non-degenerate coalgebra automorphism that induces a non-degenerate solution $r_0\colon X\times X \longrightarrow X \times X$ of the braid equation. In~\cite{GGV1}*{Proposition~4.3} we give the equations that the coefficients $\lambda_{a|b|c|d}^{e|f|g|h}$'s must satisfy in order that $r$ is a solution of the braid equation. In Corol\-lary~\ref{extremales y altura 1 son suficientes} we prove that in fact it suffices to solve a relatively small subset of these equations, corresponding to lower extremal inclusions (see Definition~\ref{definicion extremal}). For instance, when $X=\{x,y\}$ with $x<y$, then by \cite{GGV1}*{Corollary 2.5}, necessarily $r_0$ is the flip and the number of equations we must solve according to~\cite{GGV1}*{Proposition~4.3} is $125$. From these $8$ are trivially true and $36$ are solved by a general result in~\cite{GGV1}. Our result shows that it suffices to solve $7$ of the remaining $81$ equations.

Although the general problem seems to be difficult even with this reduction, the above mentioned result allows us in Section~\ref{seccion A family of examples} to make significant progress towards the solution of the following problem:
\begin{quote}
Given $r_0$ as above, find all the non-degenerate coalgebra automorphisms $r\colon D\ot D \longrightarrow D\ot D$ fulfilling the following conditions:  it is a solution of the braid equation, it induces $r_0$ on $X\times X$ and has set-type square up to height~$1$ (see Definition~\ref{def de set type square up to height 1}), and then determine which ones of these maps have set-type square (see Definition~\ref{def de set type square}).
\end{quote}
In this section we consider a height $1$ poset $(X,\le)$ with cardinal $u+v$, having $u$ minimal elements $a_0,\dots, a_{u-1}$ and $v$ maximal elements $b_0,\dots, b_{v-1}$ such that $a_i<b_j$ for all $i,j$. We assume that $u$ and $v$ are coprime, and we consider a non-degenerate bijective set-theoretic solution $r_0$ of the braid equation on $X$. Moreover, we also assume that there exist poset automorphisms $\phi_r$ and $\phi_l$ of $X$ such that
$$
r_0(x,y)=(\phi_l(y),\phi_r(x))
$$
and $\phi_r\circ \phi_l$ induces an $uv$-cycle on the set of all the pairs $(a_i,b_j)$. Our main result is Theorem~\ref{teorema principal}, in which we determine all non-degenerate coalgebra automorphisms of $D\ot D$ with set-type square up to height~$1$, that are solutions of the braid equation and induce $r_0$ on $X\times X$. This gives various infinite families of solutions of the braid equation. Finally, in Proposition~\ref{cuando es de cuadrado conjuntista}, we determine which ones of these solutions have set-type square.

\section{Preliminaries}\label{Seccion Preliminares}
A {\em partially ordered set} or {\em poset} is a pair $(X,\le)$ consisting of a set $X$ endowed with a binary relation~$\le$, called {\em an order}, that is reflexive, antisymmetric and transitive.  A {\em connected component} of $X$ is an equivalence class of the equivalence relation generated by the relation $x\sim y$ if $x$ and $y$ are comparable. The {\em height} of a finite chain $a_0<\cdots < a_n$ is $n$. The {\em height} $\mathfrak{h}(X)$ of a finite poset $X$ is the height of its largest chain. Let $a,b\in X$. The {closed interval} $[a,b]$ is the set of all the elements $c$ of $X$ such that $a\le c \le b$. We say that~$b$ covers $a$, and we write $a\prec b$ (or $b\succ a$), if $[a,b]=\{a,b\}$. A poset $X$ is {\em locally finite} if $[a,b]$ is finite for all $a,b\in X$.

In the sequel $(X,\le)$ is a locally finite poset and $Y\coloneqq\{(a,b)\in X\times X: a\le b\}$. It is well known that $D\coloneqq KY$ is a counitary coalgebra, called the {\em incidence coalgebra of $X$}, via
$$
\Delta(a,b)\coloneqq \sum_{c\in [a,b]} (a,c)\ot (c,b)\quad\text{and}\quad \epsilon(a,b)=\delta_{ab}.
$$
Consider $KX$ endowed with the coalgebra structure determined by requiring that each $x\in X$ is a group like element. The $K$-linear map $\iota \colon KX \to D$ defined by $\iota(x)\coloneqq (x,x)$ is an injective coalgebra morphism, whose image is the subcoalgebra of $D$ spanned by its group like elements.

Recall from~\cite{ESS} that a map $r_0\colon X\times X\longrightarrow X\times X$ is called {\em non-degenerate} if the maps ${}^a\!(-)$ and $(-)\hs^b$ from $X$ to $X$, defined by $({}^a\!b,a\hs^b)\coloneqq r_0(a,b)$ are bijective for all $a,b\in X$.

Let $r$ be a coalgebra automorphism of $D\ot D$ and let
$$
\sigma\coloneqq (D\ot \epsilon)\circ r\quad\text{and} \quad \tau\coloneqq (\epsilon \ot D)\circ r.
$$
We say that $r$ is {\em non-degenerate} if the maps
$$
(D \ot \sigma)\circ (\Delta \ot D)\quad\text{and}\quad (\tau \ot D)\circ (D\ot \Delta)
$$
 are isomorphisms (see~\cite{GGV1}*{Subsection~1.1}).

Let $r \colon D\ot D \longrightarrow D\ot D$ be a linear map and let
$$
\bigl(\lambda_{a|b|c|d}^{e|f|g|h}\bigr)_{(a,b),(c,d),(e,f),(g,h)\in Y}
$$
be as in equality~\eqref{def de los lambdas}. In~\cite{GGV1}*{Section~2 and Theorem~3.4} we prove that $r$ is a non-degenerate coalgebra automorphism if and only if it induces by restriction a non-degenerate bijection $r_0 \colon X\times X \longrightarrow X\times X$ and
\begin{enumerate}[itemsep=0.7ex, topsep=0.7ex, label=\arabic*)]

\item for $a\le b$ and $c\le d$,
\begin{equation}\label{conjuntista}
\qquad\quad \sum_{e,g} \lambda_{a|b|c|d}^{e|e|g|g}=\delta_{ab}\delta_{cd};
\end{equation}

\item the maps ${}^a\!(-)$ and $(-)\hs^b$ are automorphisms of posets;

\item if $a$ and $b$ belong to the same component of $X$, then ${}^a\!(-)={}^b\!(-)$ and $(-)\hs^a=(-)\hs^b$;

\item if $a\le b$, $c\le d$, $e\le f$, $g\le h$ and $\lambda_{a|b|c|d}^{e|f|g|h}\ne 0$, then $a\hs^c\le g\le h\le b\hs^c$ and ${}^a\!c\le e \le f \le {}^a\!d$;

\item if $a\le b$, $c\le d$ $e\le f$, $g\le h$, $a\hs^c\le g\le h\le b\hs^c$ and ${}^a\!c\le e \le f \le {}^a\!d$, then
\begin{equation} \label{split}
\quad\qquad \lambda^{e|f|g|h}_{a|b|c|d} = \lambda_{a | z\hs^{\bar{c}} | c | {}^{\bar{a}}\!y}^{e|y|g|z}
\lambda_{z\hs^{\bar{c}}| b | {}^{\bar{a}}\!y | d}^{y|f|z|h}
\end{equation}
for each $y,z\in X$ such that $e\le y \le f$ and $g\le z \le h$.

\end{enumerate}
By \cite{GGV1}*{Remark~2.1}, we know that $\lambda_{x|x|y|y}^{{}^x \hf y|{}^x \hf y|x^y|y^x}=1$ for all $x,y\in X$. We will use freely this fact.

\section{Factorization of solutions}
Let $r \colon D\ot D \longrightarrow D\ot D$ be a non-degenerate coalgebra automorphism that induces a non-degenerate solution $r_0 \colon X \times X \to X\times X$ of the braid equation.

\smallskip

Let $(a,b),(c,d),(e,f),(g,h),(i,j),(k,l)\in Y$ and let
\begin{equation*}
T\coloneqq [a,b]\times[c,d]\times[e,f]\quad\text{and}\quad S\coloneqq [g,h]\times[i,j]\times[k,l].
\end{equation*}
We consider $X\times X\times X$ endowed with the product order. Note that $S$ and $T$ are the closed intervals $[(g,i,j),(h,k,l)]$ and $[(a,c,e),(b,d,f)]$ in $X\times X\times X$. Clearly $S\subseteq T$ if and only if
\begin{equation}\label{inclusiones}
[g,h]\subseteq[a,b],\quad [i,j]\subseteq[c,d]\quad\text{and}\quad [k,l]\subseteq[e,f].
\end{equation}
Note also that
$$
\mathfrak{h}(T)= \mathfrak{h}([a,b])+\mathfrak{h}([c,d])+\mathfrak{h}([e,f]).
$$
For $S\subseteq T$ as above we define
$$
\LBE(S,T)\coloneqq \sum_{\substack{x\in [a,g] \\ y\in [h,b]}} \sum_{\substack{w\in [c,i] \\ z\in [j,d]}} \sum_{\substack{u\in [e,k] \\ v\in [l,f]}} \lambda_{a|b|c|d}^{{}^a\hf w | {}^a\hf z | x\hs^c | y\hs^c} \lambda_{x\hs^c|y\hs^c|e|f}^{{}^{x\hs^c}\hg u | {}^{x\hs^c}\hg v|g\hs^c\hs^e | h\hs^c\hs^e} \lambda_{{}^a\hf w|{}^a\hf z|{}^{x\hs^c}\hg u| {}^{x\hs^c}\hg v}^{{}^a\hf {}^c\hf k | {}^a\hf {}^c\hf l  | {}^a\hf i \hs^{{}^{a\hs^i}\hg e}|{}^a\hf j\hs^{{}^{a\hs^j} \hg e}}
$$
and
$$
\RBE(S,T) \coloneqq \sum_{\substack{x\in [a,g] \\ y\in [h,b]}} \sum_{\substack{w\in [c,i] \\ z\in [j,d]}} \sum_{\substack{u\in [e,k] \\ v\in [l,f]}} \lambda_{c|d|e|f}^{{}^c\hf u | {}^c\hf v | w\hs^e | z\hs^e}  \lambda_{a|b|{}^c\hf u|{}^c\hf v}^{{}^a\hf {}^c\hf k | {}^a\hf {}^c\hf l |  x\hl^{{}^c\hf u} | y\hl^{{}^c\hf u}} \lambda_{x\hl^{{}^c\hf u} | y\hl^{{}^c\hf u}|w\hs^e | z\hs^e}^{{}^{a\hl^{{}^i\hf e}}\hf {i\hs^e}| {}^{a\hl^{{}^j\hf e}}\hf {j\hs^e} | g\hs^c\hs^e | h\hs^c\hs^e}.
$$
In~\cite{GGV1}*{Proposition~4.3} the following result is proved:

\begin{proposition}\label{condicion para braided}
The map $r$ is a solution of the braid equation if and only if
\begin{equation}\label{eq braided}
\LBE(S,T) = \RBE(S,T)\quad \text{for all $S\subseteq T$.}
\end{equation}
\end{proposition}

For each $S\coloneqq [(g,i,j),(h,k,l)]$, we set
$$
\psi(S)\coloneqq ({}^g\hf {}^i\hf k , {}^g\hf {}^i\hf l) \ot ( {}^{g\hl^{{}^i\hf k}}\hf {i\hs^k}, {}^{g\hl^{{}^i\hf k}}\hf {j\hs^k}) \ot (g\hs^i\hs^k ,h\hs^i\hs^k).
$$
Let $\mathfrak{T}\coloneqq (a,b)\otimes(c,d)\otimes(e,f)$. A direct computation shows that
$$
(r\ot D)\circ (D\ot r)\circ (r\ot D) (\mathfrak{T})=\sum_{S\subseteq T}\LBE(S,T) \psi(S)
$$
and
$$
(D \ot r) \circ (r\ot D)\circ (D\ot r)(\mathfrak{T})=\sum_{S\subseteq T}\RBE(S,T) \psi(S)
$$
(see the proof of~\cite{GGV1}*{Proposition~4.3}). Since $r\ot D$ and $D\ot r$ are coalgebra mor\-phisms, this implies that
\begin{equation}\label{Suma de LBE es cero}
\delta_{ab}\delta_{cd}\delta_{ef}=(\epsilon \otimes \epsilon \otimes \epsilon) (\mathfrak{T}) = \sum_{\substack{S\subseteq T\\ \mathfrak{h}(S)=0}} \LBE(S,T),
\end{equation}
and similarly
\begin{equation}\label{Suma de RBE es cero}
\delta_{ab}\delta_{cd}\delta_{ef}=\sum_{\substack{S\subseteq T\\ \mathfrak{h}(S)=0}} \RBE(S,T).
\end{equation}
Assume $S\subseteq T$ and let $(p,q,s)\in S$. We define the splitting of the inclusion $S\subseteq T$ at $(p,q,s)$ as the pair $(S_1\subseteq T_1,S_2\subseteq T_2)$, where
\begin{align*}
  &S_1\coloneqq [(g,i,j),(p,q,s)], &&T_1\coloneqq [(a,c,e),(p,q,s)], \\
  &S_2\coloneqq [(p,q,s),(h,k,l)], &&T_2\coloneqq [(p,q,s),(b,d,f)].
\end{align*}

\begin{theorem}\label{igualdad a izquierda y derecha} The following equalities hold:
\begin{align*}
& \LBE(S,T)=\LBE(S_1,T_1)\LBE(S_2,T_2)
\shortintertext{and}
& \RBE(S,T)=\RBE(S_1,T_1)\RBE(S_2,T_2).
\end{align*}
\end{theorem}
\begin{proof}
Since by definition we have
\begin{align*}
&\LBE(S,T)=\sum_{\substack{x\in [a,g] \\ y\in [h,b]}} \sum_{\substack{w\in [c,i] \\ z\in [j,d]}}  \sum_{\substack{u\in [e,k] \\ v\in [l,f]}} \lambda_{a|b|c|d}^{{}^a\hf w | {}^a\hf z | x\hs^c | y\hs^c} \lambda_{x\hs^c|y\hs^c|e|f}^{{}^{x\hs^c}\hg u | {}^{x\hs^c}\hg v|g\hs^c\hs^e | h\hs^c\hs^e} \lambda_{{}^a\hf w|{}^a\hf z|{}^{x\hs^c}\hg u| {}^{x\hs^c}\hg v}^{{}^a\hf {}^c\hf k | {}^a\hf {}^c\hf l  | {}^a\hf i\hs^{{}^{a\hs^i}\hg e}|{}^a\hf j\hs^{{}^{a\hs^j}\hg e}},\\
&\LBE(S_1,T_1)=\sum_{x\in [a,g]} \sum_{w\in [c,i]} \sum_{u\in [e,k]} \lambda_{a|p|c|q}^{{}^a\hf w | {}^a\hf q | x\hs^c | p\hs^c} \lambda_{x\hs^c|p\hs^c|e|s}^{{}^{x\hs^c}\hg u | {}^{x\hs^c}\hg s  | g\hs^c\hs^e | p\hs^c\hs^e} \lambda_{{}^a\hf w|{}^a\hf q|{}^{x\hs^c}\hg u| {}^{x\hs^c}\hg s}^{{}^a\hf {}^c\hf k | {}^a\hf {}^c\hf s  | {}^a\hf i\hs^{{}^{a\hs^i}\hg e}|{}^a\hf q\hs^{{}^{a\hs^q}\hg e}}\\
\shortintertext{and}
&\LBE(S_2,T_2)=\sum_{y\in [h,b]} \sum_{z\in [j,d]} \sum_{v\in [l,f]} \lambda_{p|b|q|d}^{{}^p\hf q | {}^p\hf z | p\hs^q | y\hs^q} \lambda_{p\hs^q|y\hs^q|s|f}^{{}^{p\hs^q}\hg s | {}^{p\hs^q}\hg v | p\hs^q\hs^s | h\hs^q\hs^s} \lambda_{{}^p\hf q | {}^p\hf z|{}^{p\hs^q}\hg s| {}^{p\hs^q}\hg v}^{{}^p\hf {}^q\hf s | {}^p\hf {}^q\hf l  | {}^p\hf q\hs^{{}^{p\hs^q}\hg s}|{}^p\hf j\hs^{{}^{p\hs^j}\hg s}},
\end{align*}
in order to prove the first equality, it suffices to note that, by~\cite{GGV1}*{Proposition~2.10}, \cite{GGV1}*{Corollary~2.5} and \cite{GGV1}*{Remark~4.2}, the equalities
\begin{align*}
& \lambda_{a|b|c|d}^{{}^a\hf w | {}^a\hf z | x\hs^c | y\hs^c} =\lambda_{a|p|c|q}^{{}^a\hf w | {}^a\hf q | x\hs^c | p\hs^c}  \lambda_{p|b|q|d}^{{}^a\hf q | {}^a\hf z | p\hs^c | y\hs^c}= \lambda_{a|p|c|q}^{{}^a\!w | {}^a\!q | x\hs^c | p\hs^c}  \lambda_{p|b|q|d}^{{}^p\!q | {}^p\!z | p\hs^q | y\hs^q}, \\
&  \lambda_{x\hs^c|y\hs^c|e|f}^{{}^{x\hs^c}\hg u | {}^{x\hs^c}\hg v  | g\hs^c\hs^e | h\hs^c\hs^e}=  \lambda_{x\hs^c|p\hs^c|e|s}^{{}^{x\hs^c}\hg u | {}^{x\hs^c}\hg s  | g\hs^c\hs^e | p\hs^c\hs^e}  \lambda_{p\hs^c|y\hs^c|s|f}^{{}^{x\hs^c}\hg s | {}^{x\hs^c}\hg v  | p\hs^c\hs^e | h\hs^c\hs^e} =  \lambda_{x\hs^c|p\hs^c|e|s}^{{}^{x\hs^c}\hg u | {}^{x\hs^c}\hg s  | g\hs^c\hs^e | p\hs^c\hs^e}  \lambda_{p\hs^q|y\hs^q|s|f}^{{}^{p\hs^q}\hg s | {}^{p\hs^q}\hg v | p\hs^q\hs^s | h\hs^q\hs^s}
\shortintertext{and}
&  \lambda_{{}^a\hf w|{}^a\hf z|{}^{x\hs^c}\hg u| {}^{x\hs^c}\hg v}^{{}^a\hf {}^c\hf k | {}^a\hf {}^c\hf l  | {}^a\hf i\hs^{{}^{a\hs^i}\hg e}|{}^a\hf j\hs^{{}^{a\hs^j}\hg e}}= \lambda_{{}^a\hf w|{}^a\hf q|{}^{x\hs^c}\hg u| {}^{x\hs^c}\hg s}^{{}^a\hf {}^c\hf k | {}^a\hf {}^c\hf s  | {}^a\hf i\hs^{{}^{a\hs^i}\hg e}|{}^a\hf q\hs^{{}^{a\hs^q}\hg e}} \lambda_{{}^a\hf q|{}^a\hf z|{}^{x\hs^c}\hg s| {}^{x\hs^c}\hg v}^{{}^a\hf {}^c\hf s | {}^a\hf {}^c\hf l  | {}^a\hf q\hs^{{}^{a\hs^q}\hg e}|{}^a\hf j\hs^{{}^{a\hs^j}\hg e}}= \lambda_{{}^a\hf w|{}^a\hf q|{}^{x\hs^c}\hg u| {}^{x\hs^c}\hg s}^{{}^a\hf {}^c\hf k | {}^a\hf {}^c\hf s  | {}^a\hf i\hs^{{}^{a\hs^i}\hg e}|{}^a\hf q\hs^{{}^{a\hs^q}\hg e}} \lambda_{{}^p\hf q|{}^p\hf z|{}^{p\hs^q}\hg s| {}^{p\hs^q}\hg v}^{{}^p\hf {}^q\hf s | {}^p\hf {}^q\hf l  | {}^p\hf q\hs^{{}^{p\hs^q}\hg s}|{}^p\hf j\hs^{{}^{p\hs^j}\hg s}}
\end{align*}
hold. We leave the proof of the second equality to the reader.
%
%
%
%
\end{proof}

\begin{definition}\label{definicion extremal}
We say that an inclusion of intervals $[\alpha,\beta]\subseteq [\gamma,\delta]$ with $\gamma<\delta$ is \emph{lower extremal} if $\alpha=\beta=\gamma$ and that is \emph{upper extremal} if $\alpha=\beta=\delta$.
\end{definition}

\begin{remark}
Note that $S\subseteq T$ is extremal if either the three inclusions in~\eqref{inclusiones} are lower extremal or the three inclusions are upper extremal.
\end{remark}

\begin{corollary}\label{extremales y altura 1 son suficientes}
The map $r$ is a solution of the braid equation if and only if iden\-tity~\eqref{eq braided} hold for all $S\subseteq T$ lower extremal with $\mathfrak{h}(T) \ge 1$ or $S=T$ and $\mathfrak{h}(T) = 1$.
\end{corollary}

\begin{proof} By Proposition~\ref{condicion para braided} we must show that $\LBE(S,T) = \RBE(S,T)$ for all $S\subseteq T$. When $\mathfrak{h}(T) = 0$ this is true since $r_0 \colon X \times X \to X\times X$ is a solution of the braid equation, while for $n=1$ it is true by hypothesis. Assume now that identity~\eqref{eq braided} hold for $S\subseteq T$ with $\mathfrak{h}(T)\le n$ for some $n\ge 1$ and set $T=[(a,c,e),(b,d,f)]$ with $\mathfrak{h}(T)=n+1$. Let $\underline{S}\coloneqq[(a,c,e),(a,c,e)]$ and $\overline{S}\coloneqq [(b,d,f),(b,d,f)]$. By the hypothesis, we know that~\eqref{eq braided} is satisfied for $S = \underline{S}$. Moreover, for $S\subseteq T$ with $S\notin\{\underline{S},\overline{S}\}$ there exists a splitting $(S_1\subseteq T_1,S_2\subseteq T_2)$ of $S\subseteq T$ with $\mathfrak{h}(T_1),\mathfrak{h}(T_2)<\mathfrak{h}(T)$. Hence, in this case the result follows by induction on $\mathfrak{h}(T)$, using Theorem~\ref{igualdad a izquierda y derecha}. Finally we have
\begin{align*}
\LBE(\overline{S},T) & = \sum_{\substack{S \subseteq T \\ \mathfrak{h}(S)=0}} \LBE(S,T) - \sum_{\substack{S\subseteq T,\, S\ne \overline{S}\\ \mathfrak{h}(S)=0}} \LBE(S,T)\\
& = \sum_{\substack{S \subseteq T \\ \mathfrak{h}(S)=0}} \RBE(S,T) - \sum_{\substack{S\subseteq T,\, S\ne \overline{S}\\ \mathfrak{h}(S)=0}}  \RBE(S,T)\\
&= \RBE(\overline{S},T),
\end{align*}
where in the second equality we have used equalities~\eqref{Suma de LBE es cero} and~\eqref{Suma de RBE es cero}.
\end{proof}

\subsection{Braid equation for lower extremal inclusions in height one~orders.}
Next we analyze exhaustively the meaning of equalities~\eqref{eq braided} when the order has height one, the sum of the lengths of the intervals $[a,b]$, $[c,d]$ and $[e,f]$ is greater than~$1$ and the inclusions are lower extremal:

\begin{enumerate}[itemsep=0.7ex, topsep=0.7ex, label=\arabic*)]

\item When $g=h=a\prec b$, $i=j=c\prec d$ and $e=f=k=l$, then~\eqref{eq braided} reduces~to
\begin{equation}\label{Caso 110}
\begin{split}
\qquad\sum_{\substack{y\in [a,b] \\ z\in [c,d]}} & \lambda_{a|b|c|d}^{{}^a\hf c | {}^a\hf z | a\hs^c | y\hs^c} \lambda_{a\hs^c|y\hs^c|e|e}^{{}^{a\hs^c}\hg e | {}^{a\hs^c}\hg e  | a\hs^c\hs^e | a\hs^c\hs^e} \lambda_{{}^a\hf c|{}^a\hf z|{}^{a\hs^c}\hg e| {}^{a\hs^c}\hg e}^{{}^a\hf {}^c\hf e | {}^a\hf {}^c\hf e| {}^a\hf c\hs^{{}^{a\hs^c}\hg e}|{}^a\hf c\hs^{{}^{a\hs^c}\hg e}}\\
&= \sum_{\substack{y\in [a,b] \\ z\in [c,d]}} \lambda_{c|d|e|e}^{{}^c\hf e | {}^c\hf e | c\hs^e | z\hs^e}  \lambda_{a|b|{}^c\hf e|{}^c\hf e}^{{}^a\hf {}^c\hf e | {}^a\hf {}^c\hf e |  a\hl^{{}^c\hf e} | y\hl^{{}^c\hf e}} \lambda_{a\hl^{{}^c\hf e} | y\hl^{{}^c\hf e}|c\hs^e | z\hs^e}^{{}^{a\hl^{{}^c\hf e}}\hf {c\hs^e}| {}^{a\hl^{{}^c\hf e}}\hf {c\hs^e} | a\hs^c\hs^e | a\hs^c\hs^e}.
\end{split}
\end{equation}

\item  When $g=h=a\prec b$, $i=j=c=d$ and $k=l=e\prec f$, then~\eqref{eq braided} reduces~to
\begin{equation}
\begin{split}\label{Caso 101}
\qquad\sum_{\substack{y\in [a,b] \\ v\in [e,f]}} & \lambda_{a|b|c|c}^{{}^a\hf c | {}^a\hf c | a\hs^c | y\hs^c} \lambda_{a\hs^c|y\hs^c| e|f}^{{}^{a\hs^c}\hg e | {}^{a\hs^c}\hg v| a\hs^c\hs^e | a\hs^c\hs^e} \lambda_{{}^a\hf c|{}^a\hf c|{}^{a\hs^c}\hg e| {}^{a\hs^c}\hg v}^{{}^a\hf {}^c\hf e | {}^a\hf {}^c\hf e  | {}^a\hf c\hs^{{}^{a\hs^c}\hg e}|{}^a\hf c\hs^{{}^{a\hs^c}\hg e}}\\
&=  \sum_{\substack{y\in [a,b] \\ v\in [e,f]}} \lambda_{c|c|e|f}^{{}^c\hf e | {}^c\hf v | c\hs^e | c\hs^e} \lambda_{a|b|{}^c\hf e|{}^c\hf v}^{{}^a\hf {}^c\hf e | {}^a\hf {}^c\hf e |  a\hl^{{}^c\hf e} | y\hl^{{}^c\hf e}} \lambda_{a\hl^{{}^c\hf e} | y\hl^{{}^c\hf e}|c\hs^e | c\hs^e}^{{}^{a\hl^{{}^c\hf e}}\hf {c\hs^e}| {}^{a\hl^{{}^c\hf e}}\hf {c\hs^e} | a\hs^c\hs^e | a\hs^c\hs^e}.
\end{split}
\end{equation}

\item  When $g=h=a=b$, $i=j=c\prec d$ and $k=l=e\prec f$, then~\eqref{eq braided} reduces~to
\begin{equation}
\begin{split}\label{Caso 011}
\qquad\sum_{\substack{z\in [c,d] \\ v\in [e,f]}}& \lambda_{a|a|c|d}^{{}^a\hf c | {}^a\hf z | a\hs^c | a\hs^c} \lambda_{a\hs^c|a\hs^c|e|f}^{{}^{a\hs^c}\hg e | {}^{a\hs^c}\hg v|a\hs^c\hs^e | a\hs^c\hs^e} \lambda_{{}^a\hf c|{}^a\hf z|{}^{a\hs^c}\hg e| {}^{a\hs^c}\hg v}^{{}^a\hf {}^c\hf e | {}^a\hf {}^c\hf e  | {}^a\hf c \hs^{{}^{a\hs^c}\hg e}|{}^a\hf c\hs^{{}^{a\hs^c}\hg e}}\\
&= \sum_{\substack{z\in [c,d] \\ v\in [e,f]}} \lambda_{c|d|e|f}^{{}^c\hf e | {}^c\hf v | c\hs^e | z\hs^e}  \lambda_{a|a|{}^c\hf e|{}^c\hf v}^{{}^a\hf {}^c\hf e | {}^a\hf {}^c\hf e |  a\hl^{{}^c\hf e} | a\hl^{{}^c\hf e}} \lambda_{a\hl^{{}^c\hf e} | a\hl^{{}^c\hf e}|c\hs^e | z\hs^e}^{{}^{a\hl^{{}^c\hf e}}\hf {c\hs^e}| {}^{a\hl^{{}^c\hf e}}\hf {c\hs^e} | a\hs^c\hs^e | a\hs^c\hs^e}.
\end{split}
\end{equation}

\item  When $g=h=a\prec b$, $i=j=c\prec d$ and $k=l=e\prec f$, then~\eqref{eq braided} reduces~to
\begin{equation}
\begin{split}\label{Caso 111}
\qquad\sum_{\substack{ y\in [a,b]\\ z\in [c,d]\\ v\in[e,f]}} &  \lambda_{a|b|c|d}^{{}^a\hf c | {}^a\hf z | a\hs^c | y\hs^c} \lambda_{a\hs^c|y\hs^c|e|f}^{{}^{a\hs^c}\hg e | {}^{a\hs^c}\hg v|a\hs^c\hs^e | a\hs^c\hs^e} \lambda_{{}^a\hf c|{}^a\hf z|{}^{a\hs^c}\hg e| {}^{a\hs^c}\hg v}^{{}^a\hf {}^c\hf e | {}^a\hf {}^c\hf e  | {}^a\hf c \hs^{{}^{a\hs^c}\hg e}|{}^a\hf c\hs^{{}^{a\hs^c} \hg e}}\\
&= \sum_{\substack{ y\in [a,b]\\ z\in [c,d]\\ v\in[e,f]}} \lambda_{c|d|e|f}^{{}^c\hf e | {}^c\hf v | c\hs^e | z\hs^e}  \lambda_{a|b|{}^c\hf e|{}^c\hf v}^{{}^a\hf {}^c\hf e | {}^a\hf {}^c\hf e |  a\hl^{{}^c\hf e} | y\hl^{{}^c\hf e}} \lambda_{a\hl^{{}^c\hf e} | y\hl^{{}^c\hf e}|c\hs^e | z\hs^e}^{{}^{a\hl^{{}^c\hf e}}\hf {c\hs^e}| {}^{a\hl^{{}^c\hf e}}\hf {c\hs^e} | a\hs^c\hs^e | a\hs^c\hs^e}.
\end{split}
\end{equation}
\end{enumerate}

\begin{theorem}\label{EYB para longitud mayor que uno}
Assume that $(X,\le)$ has height one and that equality~\eqref{eq braided} hold for all $S\subseteq T$ with $\mathfrak{h}(T) = 1$, and $S\subseteq T$ lower extremal or $S=T$. Then $r$ is a solution of the braid equation if and only if
\begin{itemize}[itemsep=0.7ex, topsep=0.7ex]

\item[-] for all $a\prec b$, $c\prec d$ and $e\in X$, the equality~\eqref{Caso 110} is satisfied,

\item[-] for all $a\prec b$, $c\in X$ and $e\prec f$, the equality~\eqref{Caso 101} is satisfied,

\item[-] for all $a\in X$, $c\prec d$ and $e\prec f$, the equality~\eqref{Caso 011} is satisfied,

\item[-] for all $a\prec b$, $c\prec d$ and $e\prec f$, the equality~\eqref{Caso 111} is satisfied.

\end{itemize}
\end{theorem}

\begin{proof} By Corollary~\ref{extremales y altura 1 son suficientes}.
\end{proof}

\section{Conditions for solutions with set-type square}
\label{seccion set type square}

Let $r \colon D\ot D \longrightarrow D\ot D$ be a non-degenerate coalgebra automorphism and let $r_0\colon X\times X \longrightarrow X \times X$ be the map induced by $r$. In the sequel we assume that $r_0$ is a non-degenerate solution of the set-theoretical braid equation and we set $r^2\coloneqq r\circ r$ and $Y\otimes Y\coloneqq \{a\otimes b : a,b\in Y \}$.

\begin{definition}\label{def de set type square}
We say that $r$ has \emph{set-type square} if $r^2(Y\ot_k Y)\subseteq Y\ot_k Y$.
\end{definition}

\begin{remark}\label{r square} By~\cite{GGV1}*{Remark~3.1}, the map $r$ has set-type square if and only if
\begin{equation}\label{r square condicion}
r^2((a,b)\otimes (c,d))=\bigl({}^{{}^{a}\hf c}\hspace{1pt} a^{\hf c},{}^{{}^{a}\hf c}\hspace{1pt} b^{\hf c}\bigr) \ot \bigl({}^a\hf c \hspace{2pt}  {}^{a\hs^c}, {}^a\hf d \hspace{2pt} {}^{a\hs^c}\bigr),
\end{equation}
for all $(a,b),(c,d)\in Y$. Consequently, if $r$ has set-type square, then $r$ permutes the elements of $Y\ot_k Y$.
\end{remark}

\begin{remark}\label{r square para altura 0} When $a=b$ and $c=d$, then equality~\eqref{r square condicion} holds since $r_0$ is a solution of the set-theoretical braid equation.
\end{remark}

\begin{definition}\label{def de set type square up to height 1}
We say that $r$ has \emph{set-type square up to height~$1$} if equality~\eqref{r square condicion} holds for all the $(a,b), (c,d)\in Y$ with $\mathfrak{h}([a,b])+\mathfrak{h}([c,d]) = 1$.
\end{definition}

Let $(a,b),(c,d),(e,f),(g,h)\in Y$ and let
\begin{equation}\label{S y T otro}
T\coloneqq [a,b]\times[c,d]\quad\text{and}\quad S\coloneqq [e,f]\times[g,h].
\end{equation}
We consider $X\times X$ endowed with the product order. Note that $S$ and $T$ are the closed intervals $[(e,g),(f,h)]$ and $[(a,c),(b,d)]$ in $X\!\times\! X$. Clearly $S\!\subseteq T$ if and only~if
\begin{equation*}
[e,f]\subseteq [a,b]\quad\text{and}\quad [g,h]\subseteq[c,d].
\end{equation*}
Note also that $\mathfrak{h}(T)= \mathfrak{h}([a,b])+\mathfrak{h}([c,d])$. For $S\subseteq T$ as above we define
\begin{align*}
&\LSQ(S,T)\coloneqq\sum_{\substack{x\in [a,e] \\ y\in [f,b]}}\sum_{\substack{w\in [c,g] \\ z\in[h,d]}} \lambda_{a|b|c|d}^{{}^a\hf w | {}^a\hf z | x\hs^c | y\hs^c} \lambda^{{}^{{}^{a}\hf w}\hspace{0pt}  e^{\hf c}|{}^{{}^{a}\hf w}\hspace{0pt} f^{\hf c}|{}^a\hf g \hspace{1pt}  {}^{x\hs^c}|{}^a\hf h \hspace{1pt} {}^{x\hs^c}}_{{}^a\hf w | {}^a\hf z | x\hs^c | y\hs^c}
\shortintertext{and}
&\RSQ(S,T) \coloneqq  \delta_{ae}\delta_{bf}\delta_{cg}\delta_{dh}.
\end{align*}
Let $\mathfrak{V}\coloneqq (a,b)\otimes(c,d)$. Applying twice~\cite{GGV1}*{Corollary~2.9}, we obtain
\begin{equation*}
r^2(\mathfrak{V}) =\sum_{\substack{[x,y]\subseteq [a,b] \\ [w,z]\subseteq [c,d]}}\sum_{\substack{[e,f]\subseteq [x,y] \\ [g,h]\subseteq [w,z]}} \lambda_{a|b|c|d}^{{}^a\hf w | {}^a\hf z | x\hs^c | y\hs^c} \lambda^{{}^{{}^{a}\hf w}\hspace{0pt}  e^{\hf c}|{}^{{}^{a}\hf w}\hspace{0pt} f^{\hf c}|{}^a\hf g \hspace{1pt}  {}^{x\hs^c}|{}^a\hf h \hspace{1pt} {}^{x\hs^c}}_{{}^a\hf w | {}^a\hf z | x\hs^c | y\hs^c} \bigl({}^{{}^{a}\hf w}\hspace{0pt}  e^{\hf c},{}^{{}^{a}\hf w}\hspace{0pt} f^{\hf c}\bigr) \ot \bigl({}^a\hf g \hspace{1pt}  {}^{x\hs^c}, {}^a\hf h \hspace{1pt} {}^{x\hs^c}\bigr).
\end{equation*}
Consequently, since by~\cite{GGV1}*{Corollary~2.5}
$$
\bigl({}^{{}^{a}\hf w}\hspace{0pt}  e^{\hf c},{}^{{}^{a}\hf w}\hspace{0pt} f^{\hf c}\bigr) \ot \bigl({}^a\hf g \hspace{1pt}  {}^{x\hs^c}, {}^a\hf h \hspace{1pt} {}^{x\hs^c}\bigr) = \bigl({}^{{}^{a}\hf c}\hspace{0pt}  e^{\hf c},{}^{{}^{a}\hf c}\hspace{0pt} f^{\hf c}\bigr) \ot \bigl({}^a\hf g \hspace{1pt}  {}^{a\hs^c}, {}^a\hf h \hspace{1pt} {}^{a\hs^c}\bigr),
$$
we have
\begin{equation}\label{ecuacion basica}
r^2(\mathfrak{V})=\sum_{S\subseteq T}\LSQ(S,T) \phi(S),
\end{equation}
where for $S\coloneqq [e,f]\times[g,h]$ we set $\phi(S)\coloneqq \bigl({}^{{}^{a}\hf c}\hspace{0pt}  e^{\hf c},{}^{{}^{a}\hf c}\hspace{0pt} f^{\hf c}\bigr) \ot \bigl({}^a\hf g \hspace{1pt}  {}^{a\hs^c}, {}^a\hf h \hspace{1pt} {}^{a\hs^c}\bigr)$. Since $r^2$ is a coalgebra morphism, this implies that
\begin{equation*}
\delta_{ab}\delta_{cd}=(\epsilon \otimes \epsilon) (\mathfrak{V}) = \sum_{\substack{S\subseteq T\\ \mathfrak{h}(S)=0}} \LSQ(S,T).
\end{equation*}
Let $S\subseteq T$ and let $(p,q)\in S$. We define the splitting of the inclusion $S\subseteq T$ at $(p,q)$ as the pair $(S_1\subseteq T_1,S_2\subseteq T_2)$, where
\begin{align*}
&S_1\coloneqq [(e,g),(p,q)], &&T_1\coloneqq [(a,c),(p,q)], \\
&S_2\coloneqq [(p,q),(f,h)], &&T_2\coloneqq [(p,q),(b,d)].
\end{align*}
Note that
\begin{equation}\label{particion de RSQ}
\RSQ(S,T)=\RSQ(S_1,T_1)\RSQ(S_2,T_2)
\end{equation}
for each splitting of $S\subseteq T$.

\begin{proposition}\label{prop quasi involutivo} Equality~\eqref{r square condicion} is true for $(a,b)\ot (c,d)$ if and only if
\begin{equation}\label{relacion de alphas}
\LSQ(S,T) = \RSQ(S,T)\qquad\text{for all $S\subseteq T$,}
\end{equation}
where $T\coloneqq [a,b]\times [c,d]$.
\end{proposition}

\begin{proof} Since the map
$$
(e,f)\ot (g,h)\mapsto \bigl({}^{{}^{a}\hf c}\hspace{0pt}  e^{\hf c},{}^{{}^{a}\hf c}\hspace{0pt} f^{\hf c}\bigr) \ot \bigl({}^a\hf g \hspace{1pt}  {}^{a\hs^c}, {}^a\hf h \hspace{1pt} {}^{a\hs^c}\bigr)
$$
is injective, the result follows comparing coefficients in equalities~\eqref{r square condicion} and~\eqref{ecuacion basica}.
\end{proof}

\begin{proposition}\label{igualdad a izquierda y derecha'} Let $S\subseteq T$, $(p,q)\in S$ and let $((S_1,T_1),(S_2,T_2))$ be
the splitting of $S\subseteq T$ at $(p,q)$. The following equality hold:
$$
\LSQ(S,T)=\LSQ(S_1,T_1)\LSQ(S_2,T_2).
$$

\end{proposition}

\begin{proof} Since, by definition
\begin{align*}
&\LSQ(S,T)=\sum_{\substack{x\in [a,e] \\ y\in [f,b]}}\sum_{\substack{w\in [c,g] \\ z\in[h,d]}} \lambda_{a|b|c|d}^{{}^a\hf w | {}^a\hf z | x\hs^c | y\hs^c} \lambda^{{}^{{}^{a}\hf w}\hspace{0pt}  e^{\hf c}|{}^{{}^{a}\hf w}\hspace{0pt} f^{\hf c}|{}^a\hf g \hspace{1pt}  {}^{x\hs^c}|{}^a\hf h \hspace{1pt} {}^{x\hs^c}}_{{}^a\hf w | {}^a\hf z | x\hs^c | y\hs^c},\\
&\LSQ(S_1,T_1)=\sum_{x\in [a,e] }\sum_{w\in [c,g]} \lambda_{a|p|c|q}^{{}^a\hf w | {}^a\hf q | x\hs^c | p\hs^c} \lambda^{{}^{{}^{a}\hf w}\hspace{0pt}  e^{\hf c}|{}^{{}^{a}\hf w}\hspace{0pt} p^{\hf c}|{}^a\hf g \hspace{1pt}  {}^{x\hs^c}|{}^a\hf q \hspace{1pt} {}^{x\hs^c}}_{{}^a\hf w | {}^a\hf q | x\hs^c | p\hs^c}\\
\shortintertext{and}
&\LSQ(S_2,T_2)=\sum_{y\in [f,b]}\sum_{z\in[h,d]} \lambda_{p|b|q|d}^{{}^p\hf q | {}^p\hf z | p\hs^q | y\hs^q} \lambda^{{}^{{}^p\hf q}\hspace{0pt}  p^{\hf q}|{}^{{}^p\hf q}\hspace{0pt} f^{\hf q}|{}^p\hf q \hspace{1pt}  {}^{p\hs^q}|{}^p\hf h \hspace{1pt} {}^{p\hs^q}}_{{}^p\hf q | {}^p\hf z | p\hs^q | y\hs^q},
\end{align*}
in order to prove the first equality, it suffices to note that, by~\cite{GGV1}*{Proposition~2.10} and \cite{GGV1}*{Corollary~2.5}, the equalities
\begin{align*}
\lambda_{a|b|c|d}^{{}^a\hf w | {}^a\hf z | x\hs^c | y\hs^c} & =\lambda_{a|p|c|q}^{{}^a\hf w | {}^a\hf q | x\hs^c | p\hs^c}  \lambda_{p|b|q|d}^{{}^a\hf q | {}^a\hf z | p\hs^c | y\hs^c}= \lambda_{a|p|c|q}^{{}^a\!w | {}^a\!q | x\hs^c | p\hs^c}  \lambda_{p|b|q|d}^{{}^p\!q | {}^p\!z | p\hs^q | y\hs^q}\\
\shortintertext{and}
\lambda^{{}^{{}^{a}\hf w}\hspace{0pt}  e^{\hf c}|{}^{{}^{a}\hf w}\hspace{0pt} f^{\hf c}|{}^a\hf g \hspace{1pt}  {}^{x\hs^c}|{}^a\hf h \hspace{1pt} {}^{x\hs^c}}_{{}^a\hf w | {}^a\hf z | x\hs^c | y\hs^c} & = \lambda^{{}^{{}^{a}\hf w}\hspace{0pt}  e^{\hf c}|{}^{{}^{a}\hf w}\hspace{0pt} p^{\hf c}|{}^a\hf g \hspace{1pt}  {}^{x\hs^c}|{}^a\hf q \hspace{1pt} {}^{x\hs^c}}_{{}^a\hf w | {}^a\hf q | x\hs^c | p\hs^c} \lambda^{{}^{{}^a\hf w}\hspace{0pt}  p^{\hf c}|{}^{{}^a\hf w}\hspace{0pt} f^{\hf c}|{}^a\hf q \hspace{1pt}  {}^{x\hs^c}|{}^a\hf h \hspace{1pt} {}^{x\hs^c}}_{{}^a\hf q | {}^a\hf z | p\hs^c | y\hs^c}\\
& = \lambda^{{}^{{}^{a}\hf w}\hspace{0pt}  e^{\hf c}|{}^{{}^{a}\hf w}\hspace{0pt} p^{\hf c}|{}^a\hf g \hspace{1pt}  {}^{x\hs^c}|{}^a\hf q \hspace{1pt} {}^{x\hs^c}}_{{}^a\hf w | {}^a\hf q | x\hs^c | p\hs^c} \lambda^{{}^{{}^p\hf q}\hspace{0pt}  p^{\hf q}|{}^{{}^p\hf q}\hspace{0pt} f^{\hf q}|{}^p\hf q \hspace{1pt}  {}^{p\hs^q}|{}^p\hf h \hspace{1pt} {}^{p\hs^q}}_{{}^p\hf q | {}^p\hf z | p\hs^q | y\hs^q}
\end{align*}
hold.
\end{proof}

As in Definition~\ref{definicion extremal} we say that $[\alpha,\beta]\subseteq [\gamma,\delta]$ with $\gamma<\delta$ is \emph{lower extremal} if $\alpha=\beta=\gamma$ and that it is \emph{upper extremal} if $\alpha=\beta=\delta$.

\begin{corollary}\label{extremales y altura 1 son suficientes'} The following assertions hold:

\begin{enumerate}

\smallskip

\item The map $r$ has set-type square up to height~$1$ if and only if the equality $\LSQ(S,T) = \RSQ(S,T)$ is true for all $S\subseteq T$ such that $\mathfrak{h}(T) = 1$, and $S = T$ or $S\subseteq T$ is lower extremal.

\smallskip

\item The map $r$ has set-type square if and only if it has set-type square up to height~$1$ and $\LSQ(S,T) = \RSQ(S,T)$ for all $S\subseteq T$ lower extremal, such that $\mathfrak{h}(T)\ge 2$.

\end{enumerate}

\end{corollary}

\begin{proof} Mimic the proof of Corollary~\ref{extremales y altura 1 son suficientes} using equality~\eqref{particion de RSQ} and Proposition~\ref{igualdad a izquierda y derecha'}.
\end{proof}

For the rest of the section we will assume that $(X,\le)$ is connected.

\begin{remark}
 \label{conexo no degenerado}
By~\cite{GGV1}*{Corollary~2.5} there exist order automorphisms $\phi_l$ and $\phi_r$ of $X$ such that
$$
{}^ba=\phi_l(a)\quad\text{and}\quad a^b=\phi_r(a)\qquad\text{for all $a,b\in X$.}
$$
Moreover, since $r_0\colon X\times X \longrightarrow X \times X$ is a solution of the set-theoretic braid equation,
$\phi_l$ and $\phi_r$ commute. Consequently, $r$ has set-type square if and only if
$$
r^2((a,b)\otimes (c,d))=  \bigl(\varphi(a),\varphi(b)\bigr)\otimes \bigl(\varphi(c),\varphi(d)\bigr) \quad\text{for all $(a,b),(c,d)\in Y$,}
$$
where $\varphi\coloneqq\phi_l\circ \phi_r$.
\end{remark}

\begin{notation}\label{notacion alphas}
For all $s,a,b,c,d\in X$ with $a\prec b$, $c\prec d$ and $i\in \mathds{Z}$, we will write
\begin{alignat}{2}
&s^{(i)}\coloneqq \phi_r^{i}(s),&&\qquad {}^{(i)}\hf s \coloneqq \phi_l^i(s),\label{potencia de phi}\\
&\alpha_r(s)(a,b)\coloneqq\lambda_{a|b|s|s}^{{}^{(1)}\hf s |{}^{(1)}\hf s| a^{(1)} | b^{(1)}},&&\qquad \beta_r(s)(a,b)\coloneqq\lambda_{a|b|s|s}^{{}^{(1)}\hf s |{}^{(1)}\hf s| a^{(1)} | a^{(1)}},\label{definicion de a sub r}\\
&\alpha_l(s)(a,b)\coloneqq\lambda_{s|s|a|b}^{ {}^{(1)}\hf a | {}^{(1)}\hf b | s^{(1)} | s^{(1)}},&&\qquad
\beta_l(s)(a,b)\coloneqq\lambda_{s|s|a|b}^{{}^{(1)}\hf a | {}^{(1)}\hf a | s^{(1)} |s^{(1)}},\label{definicion de a sub l}\\
&\Gamma_{a|b|c|d}\coloneqq \lambda_{a|b|c|d}^{{}^{(1)}\hf c|{}^{(1)}\hf c|a^{(1)}|a^{(1)}}.
\end{alignat}
\end{notation}

\begin{proposition}\label{alpha y beta en quasi involutivo} The map $r$ has set-type square up to height~$1$ if and only if for all $a\prec b$ and $c\in X$
\begin{align}\label{relacion de alphas y betas 1}
&\alpha_r(c)(a,b)\alpha_l({}^{(1)}c)(a^{(1)},b^{(1)})=1,\\
\label{relacion de alphas y betas 2}
&\alpha_r(c)(a,b)\beta_l({}^{(1)}c)(a^{(1)},b^{(1)})+\beta_r(c)(a,b)=0,\\
\label{relacion de alphas y betas 3}
&\alpha_l(c)(a,b)\alpha_r(c^{(1)})({}^{(1)}a,{}^{(1)}b)=1,\\
\shortintertext{and}
\label{relacion de alphas y betas 4}
&\alpha_l(c)(a,b)\beta_r(c^{(1)})({}^{(1)}a,{}^{(1)}b)+\beta_l(c)(a,b)=0.
\end{align}
\end{proposition}

\begin{proof} When $a\prec b$, $e=a$, $f=b$ and $c=d=g=h$, equality~\eqref{relacion de alphas} becomes
$$
\lambda_{a|b|c|c}^{{}^{(1)}\hf c | {}^{(1)}\hf c | a\hs^{(1)} | b\hs^{(1)}} \lambda_{{}^{(1)}\hf c | {}^{(1)}\hf c | a\hs^{(1)} | b\hs^{(1)}}^{{}^{(1)}\hf a\hf^{(1)}|{}^{(1)}\hf b\hf^{(1)}|{}^{(1)}\hf c\hf^{(1)}|{}^{(1)}\hf c\hf^{(1)}}=1,
$$
which coincides with equality~\eqref{relacion de alphas y betas 1}. Similarly, when $a=b=e=f$, $c\prec d$, $c=d$ and $g=h$, equality~\eqref{relacion de alphas} reduce to equality~\eqref{relacion de alphas y betas 3}. On the other hand, when $e=f=a\prec b$ and $c=d=g=h$ equality~\eqref{relacion de alphas} gives
$$
\lambda_{a|b|c|c}^{{}^{(1)}\hf c | {}^{(1)}\hf c | a\hs^{(1)} | a\hs^{(1)}} \lambda_{{}^{(1)}\hf c | {}^{(1)}\hf c | a\hs^{(1)} | a\hs^{(1)}}^{{}^{(1)}\hf a\hf^{(1)}|{}^{(1)}\hf a\hf^{(1)}|{}^{(1)}\hf c\hf^{(1)}|{}^{(1)}\hf c\hf^{(1)}} + \lambda_{a|b|c|c}^{{}^{(1)}\hf c | {}^{(1)}\hf c | a\hs^{(1)} | b\hs^{(1)}} \lambda_{{}^{(1)}\hf c | {}^{(1)}\hf c | a\hs^{(1)} | b\hs^{(1)}}^{{}^{(1)}\hf a\hf^{(1)}|{}^{(1)}\hf a\hf^{(1)}|{}^{(1)}\hf c\hf^{(1)}|{}^{(1)}\hf c\hf^{(1)}}=0,
$$
which coincides with equality~\eqref{relacion de alphas y betas 2}. A similar computation shows that when $a=b=e=f$ and $g=h=c \prec d$, equality~\eqref{relacion de alphas} reduce to equality~\eqref{relacion de alphas y betas 4}. By Corollary~\ref{extremales y altura 1 son suficientes'}(1) this finishes the proof.
\end{proof}

\begin{corollary}\label{alfas y betas periodicidad}
If $r$ has set-type square up to height~$1$, then
$$
\alpha_h({}^{(\hspace{-0.35pt} 1 \hspace{-0.35pt})}\hf c^{(\hspace{-0.35pt} 1 \hspace{-0.35pt})})({}^{(\hspace{-0.35pt} 1 \hspace{-0.35pt})}\hf a^{(\hspace{-0.35pt} 1 \hspace{-0.35pt})},{}^{(\hspace{-0.35pt} 1 \hspace{-0.35pt})}\hf b^{(\hspace{-0.35pt} 1 \hspace{-0.35pt})})=\alpha_l(c)(a,b)\hspace{-0.5pt}\quad\text{and}\hspace{-0.5pt}\quad \beta_h({}^{(\hspace{-0.35pt} 1 \hspace{-0.35pt})}\hf c^{(\hspace{-0.35pt} 1 \hspace{-0.35pt})})({}^{(\hspace{-0.35pt} 1 \hspace{-0.35pt})}\hf a^{(\hspace{-0.35pt} 1 \hspace{-0.35pt})},{}^{(\hspace{-0.35pt} 1 \hspace{-0.35pt})}\hf b^{(\hspace{-0.35pt} 1 \hspace{-0.35pt})})=\beta_h(c)(a,b),
$$
for $h\in \{r,l\}$, $a\prec b$ in $X$ and $c\in X$.
\end{corollary}

\begin{proof}
We only consider the case $h = l$ since the case $h=r$ is similar. By equalities~\eqref{relacion de alphas y betas 1} and \eqref{relacion de alphas y betas 3},
$$
\alpha_l({}^{(1)}\hf c^{(1)})({}^{(1)}\hf a^{(1)},{}^{(1)}\hf b^{(1)})=\frac{1}{\alpha_r(c^{(1)})({}^{(1)}\hf a,{}^{(1)}\hf b)} = \alpha_l(c)(a,b),
$$
which proves the first equality, and by equalities \eqref{relacion de alphas y betas 2}, \eqref{relacion de alphas y betas 3} and \eqref{relacion de alphas y betas 4},
\begin{align*}
\beta_l({}^{(1)}\hf c^{(1)})({}^{(1)}\hf a^{(1)},{}^{(1)}\hf b^{(1)}) & = \alpha_l(c)(a,b) \alpha_r(c^{(1)})({}^{(1)}\hf a, {}^{(1)}\hf b) \beta_l({}^{(1)}\hf c^{(1)})({}^{(1)}\hf a^{(1)},{}^{(1)}\hf b^{(1)})\\
& =-\alpha_l(c)(a,b) \beta_r(c^{(1)})({}^{(1)}\hf a, {}^{(1)}\hf b)\\
&=\beta_l(c)(a,b),
\end{align*}
which proves the second equality.
\end{proof}

We will need the following result that complements~\cite{GGV1}*{Proposition~4.5}.

\begin{proposition}\label{complemento de la prop 4.5 de [10]}
Item~6) of \cite{GGV1}*{Sub\-sec\-tion~4.1} is satisfied if and only if for all $a\prec b$ in $X$ there exists a constant $C_m(a,b)\in K^{\times}$ such that
\begin{equation*}\label{cocientes alpha mas condicion}
\qquad\quad \frac{\alpha_l(s\hs^{(1)})(a\hs^{(1)},b\hs^{(1)})}{\alpha_l(s)(a,b)}= \frac{\alpha_r({}^{(1)}\hf t)({}^{(1)}\hf a,{}^{(1)}\hf b)}{\alpha_r(t)(a,b)} = C_m(a,b),
\end{equation*}
for all $s,t\in X$.
\end{proposition}

\begin{proof}
Mimic the proof of ~\cite{GGV1}*{Proposition~4.5(1)}.
\end{proof}

\begin{remark}\label{condiciones para longitud menor que uno 1}
By items~1)--4) of~\cite{GGV1}*{Proposition~4.5} and Proposition~\ref{complemento de la prop 4.5 de [10]}, if $r$ is a solution of the braid equation, then necessarily
\begin{align}
   & \alpha_r(s)(a,b)=\alpha_r({}^{(1)}\hf s \hs^{(1)})(a,b),&& \alpha_l(s)(a,b)=\alpha_l({}^{(1)}\hf s \hs^{(1)})(a,b),\label{alpha constante}\\
   & \beta_r(s)(a,b)=\beta_r({}^{(1)}\hf s \hs^{(1)})(a,b),&& \beta_l(s)(a,b)=\beta_l({}^{(1)}\hf s \hs^{(1)})(a,b),\label{beta constante}
\end{align}
and for all $a\prec b$ there exist constants $C_r(a,b)$, $C_l(a,b)$ and $C_m(a,b)$ such that
\begin{align}
&C_r(a,b) = \frac{\alpha_r(s^{(1)})(a^{(1)},b^{(1)})}{\alpha_r(s)(a,b)}\quad\text{for all $s$,}\label{cociente alphar}\\
&C_l(a,b) = \frac{\alpha_l({}^{(1)}s)({}^{(1)}a,{}^{(1)}b)}{\alpha_l(s)(a,b)}\quad\text{for all $s$,}\label{cociente alphal}
\shortintertext{and}
& C_m(a,b) = \frac{\alpha_l(s^{(1)})(a^{(1)},b^{(1)})}{\alpha_l(s)(a,b)}= \frac{\alpha_r({}^{(1)}t)({}^{(1)}a,{}^{(1)}b)}{\alpha_r(t)(a,b)}\quad\text{for all $s$ and $t$.}\label{cociente alpham}
\end{align}
Assume now that $r$ has set-type square up to height~$1$. Then, by Corollary~\ref{alfas y betas periodicidad},
$$
C_r(a,b)C_m(a,b) = \frac{\alpha_r(s^{(1)})(a^{(1)},b^{(1)})}{\alpha_r(s)(a,b)} \frac{\alpha_r({}^{(1)}s^{(1)})({}^{(1)}a^{(1)},{}^{(1)}b^{(1)})}{\alpha_r(s^{(1)})(a^{(1)},b^{(1)})} = 1
$$
and, similarly, $C_l(a,b)C_m(a,b) = 1$. So, $C_r(a,b) = C_m(a,b)^{-1} = C_l(a,b)$.
\end{remark}

\begin{proposition}\label{rel a prec b y c prec d}
Assume that $(X,\le)$ has height~$1$. If $r$ has set-type square up to height~$1$, then $r$ has set-type square if and only if
\begin{equation}
\begin{aligned}\label{para altura 1}
\frac{\alpha_l(a)(c,d)}{\alpha_l({}^{(\hspace{-0.2pt} 1 \hspace{-0.2pt})}\hf c)(a^{(\hspace{-0.2pt} 1 \hspace{-0.2pt})},b^{(\hspace{-0.2pt} 1 \hspace{-0.2pt})})} \Gamma_{{}^{(\hspace{-0.2pt} 1 \hspace{-0.2pt})}\hf c|{}^{(\hspace{-0.2pt} 1 \hspace{-0.2pt})}\hf d|a^{(\hspace{-0.2pt} 1 \hspace{-0.2pt})}|b^{\hspace{-0.2pt} 1 \hspace{-0.2pt})}} = & -\Gamma_{a|b|c|d} - \frac{\beta_l(b)(c,d) \beta_l({}^{(\hspace{-0.2pt} 1 \hspace{-0.2pt})}\hf c)(a^{(\hspace{-0.2pt} 1 \hspace{-0.2pt})},b^{(\hspace{-0.2pt} 1 \hspace{-0.2pt})})}{\alpha_l({}^{(\hspace{-0.2pt} 1 \hspace{-0.2pt})}\hf c)(a^{(\hspace{-0.2pt} 1 \hspace{-0.2pt})},b^{(\hspace{-0.2pt} 1 \hspace{-0.2pt})})} \\
&  - \frac{\beta_l(a)(c,d) \beta_l({}^{(1)}\hf d)(a^{(1)},b^{(1)})}{\alpha_l({}^{(1)}\hf d)(a^{(1)},b^{(1)})},
\end{aligned}
\end{equation}
for all $a,b,c,d\in X$ with $a \prec b$ and $c\prec d$.
\end{proposition}

\begin{proof}
By Corollary~\ref{extremales y altura 1 son suficientes'}(2) and the definition of $\RSQ(S,T)$, we must show that the hypothesis in the statement is equivalent to the fact that $\LSQ(S,T) = 0$ for all~$T\coloneqq [a,b]\times [c,d]$ with $a\prec b$ and $c\prec d$ and $S\subseteq T$ lower extremal. By the very definition of $\LSQ(S,T)$, we have
\begin{align*}
\LSQ(S,T) & = \lambda_{a|b|c|d}^{{}^{(1)}\hf  c | {}^{(1)}\hf c | a^{(1)} | a^{(1)}} \lambda_{{}^{(1)}\hf  c | {}^{(1)}\hf c | a^{(1)} | a^{(1)}}^{{}^{(1)}\hf a^{(1)} |{}^{(1)}\hf a^{(1)} |{}^{(1)}\hf c^{(1)} |{}^{(1)}\hf c^{(1)}}\\[3pt]
& + \lambda_{a|b|c|d}^{{}^{(1)}\hf  c | {}^{(1)}\hf c | a^{(1)} | b^{(1)}} \lambda_{{}^{(1)}\hf  c | {}^{(1)}\hf c | a^{(1)} | b^{(1)}}^{{}^{(1)}\hf a^{(1)} |{}^{(1)}\hf a^{(1)} |{}^{(1)}\hf c^{(1)} |{}^{(1)}\hf c^{(1)}}\\[3pt]
&+ \lambda_{a|b|c|d}^{{}^{(1)}\hf  c | {}^{(1)}\hf d | a^{(1)} | a^{(1)}} \lambda_{{}^{(1)}\hf  c | {}^{(1)}\hf d | a^{(1)} | a^{(1)}}^{{}^{(1)}\hf a^{(1)} |{}^{(1)}\hf a^{(1)} |{}^{(1)}\hf c^{(1)} |{}^{(1)}\hf c^{(1)}}\\[3pt]
&+ \lambda_{a|b|c|d}^{{}^{(1)}\hf  c | {}^{(1)}\hf d | a^{(1)} | b^{(1)}} \lambda_{{}^{(1)}\hf  c | {}^{(1)}\hf d | a^{(1)} | b^{(1)}}^{{}^{(1)}\hf a^{(1)} |{}^{(1)}\hf a^{(1)} |{}^{(1)}\hf c^{(1)} |{}^{(1)}\hf c^{(1)}}\\[3pt]
& = \Gamma_{a|b|c|d} + \lambda_{a|b|c|d}^{{}^{(1)}\hf  c | {}^{(1)}\hf c | a^{(1)} | b^{(1)}} \beta_l({}^{(1)}\hf c)(a^{(1)},b^{(1)})\\[3pt]
&+ \lambda_{a|b|c|d}^{{}^{(1)}\hf  c | {}^{(1)}\hf d | a^{(1)} | a^{(1)}} \beta_r(a^{(1)})({}^{(1)}\hf c,{}^{(1)}\hf d)
+ \lambda_{a|b|c|d}^{{}^{(1)}\hf  c | {}^{(1)}\hf d | a^{(1)} | b^{(1)}} \Gamma_{{}^{(1)}\hf c|{}^{(1)}\hf d| a^{(1)}|b^{(1)}},
\end{align*}
where the last equality follows from the definitions of $\Gamma$, $\beta_l$ and $\beta_r$. Since, by~\eqref{split} and the fact that the maps ${}^a\!(-)$ and $(-)\hs^b$ are automorphisms of posets,
\begin{align*}
&\lambda_{a|b|c|d}^{{}^{(1)}\hf  c | {}^{(1)}\hf c | a^{(1)} | b^{(1)}} = \lambda_{a|b|c|c}^{{}^{(1)}\hf  c | {}^{(1)}\hf c | a^{(1)} | b^{(1)}} \lambda_{b|b|c|d}^{{}^{(1)}\hf  c | {}^{(1)}\hf c | a^{(1)} | b^{(1)}} = \alpha_r(c)(a,b)\beta_l(b)(c,d)\\
&\lambda_{a|b|c|d}^{{}^{(1)}\hf  c | {}^{(1)}\hf d | a^{(1)} | a^{(1)}} = \lambda_{a|a|c|d}^{{}^{(1)}\hf  c | {}^{(1)}\hf d | a^{(1)} | a^{(1)}} \lambda_{a|b|d|d}^{{}^{(1)}\hf  d | {}^{(1)}\hf d | a^{(1)} | a^{(1)}} = \alpha_l(a)(c,d)\beta_r(d)(a,b)
\shortintertext{and}
&\lambda_{a|b|c|d}^{{}^{(1)}\hf  c | {}^{(1)}\hf d | a^{(1)} | b^{(1)}} = \lambda_{a|b|c|c}^{{}^{(1)}\hf  c | {}^{(1)}\hf c | a^{(1)} | b^{(1)}} \lambda_{b|b|c|d}^{{}^{(1)}\hf  c | {}^{(1)}\hf d | b^{(1)} | b^{(1)}} = \alpha_r(c)(a,b)\alpha_l(b)(c,d),
\end{align*}
we obtain
\begin{align*}
\LSQ(S,T) & = \Gamma_{a|b|c|d} + \alpha_r(c)(a,b)\beta_l(b)(c,d)\beta_l({}^{(1)}\hf c)(a^{(1)}, b^{(1)})\\[3pt]
& + \alpha_l(a)(c,d)\beta_r(d)(a,b)\beta_r(a^{(1)})({}^{(1)}\hf c,{}^{(1)}\hf d)\\[3pt]
& + \alpha_r(c)(a,b)\alpha_l(b)(c,d)\Gamma_{{}^{(1)}\hf c|{}^{(1)}\hf d| a^{(1)}|b^{(1)}}\\[3pt]
& = \Gamma_{a|b|c|d} + \frac{\beta_l(b)(c,d) \beta_l({}^{(1)}\hf c)(a^{(1)}, b^{(1)})}{\alpha_l({}^{(1)}\hf c)(a^{(1)}, b^{(1)})}\\[3pt]
& + \frac{\alpha_l(a)(c,d) \beta_l({}^{(1)}\hf d)(a^{(1)}, b^{(1)}) \beta_l({}^{(1)}\hf a^{(1)})({}^{(1)}\hf c^{(1)},{}^{(1)}\hf d^{(1)})}{\alpha_l({}^{(1)}\hf d)(a^{(1)}, b^{(1)}) \alpha_l({}^{(1)}\hf a^{(1)})({}^{(1)}\hf c^{(1)},{}^{(1)}\hf d^{(1)})}\\[3pt]
& + \frac{\alpha_l(a)(c,d)}{\alpha_l({}^{(1)}\hf c)(a^{(1)}, b^{(1)})} \Gamma_{{}^{(1)}\hf c|{}^{(1)}\hf d| a^{(1)}|b^{(1)}}\\[3pt]
& = \Gamma_{a|b|c|d} + \frac{\beta_l(b)(c,d) \beta_l({}^{(1)}\hf c)(a^{(1)}, b^{(1)})}{\alpha_l({}^{(1)}\hf c)(a^{(1)}, b^{(1)})}\\[3pt]
& + \frac{\beta_l(a)(c,d) \beta_l({}^{(1)}\hf d)(a^{(1)}, b^{(1)}) }{\alpha_l({}^{(1)}\hf d)(a^{(1)}, b^{(1)})} + \frac{\alpha_l(a)(c,d)}{\alpha_l({}^{(1)}\hf c)(a^{(1)}, b^{(1)})} \Gamma_{{}^{(1)}\hf c|{}^{(1)}\hf d| a^{(1)}|b^{(1)}},
\end{align*}
where the second equality holds by Proposition~\ref{alpha y beta en quasi involutivo}; and the third one, by Corol\-lary~\ref{alfas y betas periodicidad}. Hence $\LSQ(S,T) = 0$ if and only if equality~\eqref{para altura 1} is true.
\end{proof}

\begin{corollary}
If $(X,\le)$ has height~$1$, then $r$ has set-type square if and only if equalities~\eqref{relacion de alphas y betas 1}--\eqref{para altura 1} are fulfilled.
\end{corollary}

\begin{proof} This follows immediately from Propositions~\ref{prop quasi involutivo}, \ref{alpha y beta en quasi involutivo} and \ref{rel a prec b y c prec d}.
\end{proof}

\section{A family of examples}\label{seccion A family of examples}
In this section we assume that $u,v\in\mathds{N}$ are coprime and for each $l\in\mathds{N}$ we set $\mathds{N}_l\coloneqq\{0,\dots,l-1\}$. Define on the set
$$
X=\{a_0,\dots,a_{u-1},b_0,\dots,b_{v-1}\}
$$
the partial order $a_i<b_j$ for all $i,j$. For the rest of the paper we assume that each element of~$K^{\times}$ has $uv$ distinct $uv$-th roots, and we fix a primitive $uv$-th root of unity $w$.

Let $\phi_r,\phi_l\colon X\to X$ be two commuting poset automorphisms and let
$$
r_0\colon X\times X\longrightarrow X\times X
$$
be the map defined by $r_0(x,y)\coloneqq (\phi_l(y),\phi_r(x))$. It is well-known that $r_0$ is a non-degenerate solution of the braid equation. Moreover, by definition, ${}^x\hf(-) = \phi_l$ and $(-)\hs^x = \phi_r$ for all $x$.

Since $\phi_l$ and $\phi_r$ are poset automorphisms there exist permutations $\sigma_a$, $\tau_a$ of $\mathds{N}_u$, and
$\sigma_b$, $\tau_b$ of $\mathds{N}_v$ such that
\begin{equation}\label{phi en relacion a sigma}
\phi_l(a_i)=a_{\sigma_a(i)},\quad \phi_l(b_j)=b_{\sigma_b(j)},\quad \phi_r(a_i)=a_{\tau_a(i)}\quad\text{and}\quad \phi_r(b_j)=b_{\tau_b(j)}.
\end{equation}
Note that $\sigma_a\circ \tau_a = \tau_a\circ \sigma_a$ and $\sigma_b\circ \tau_b = \tau_b\circ \sigma_b$ because $\phi_r$ and $\phi_l$ commute. From now on we assume that $\varsigma_a\coloneqq \sigma_a\circ \tau_a$ is an $u$-cycle and $\varsigma_b\coloneqq \sigma_b\circ \tau_b$ is an $v$-cycle. Since $\gcd(u,v)=1$ from this it follows that the map $(i,j)\mapsto (\varsigma_a(i), \varsigma_b(j))$ is an $uv$-cycle.

Let $r \colon D\ot D \longrightarrow D\ot D$ be a non-degenerate coalgebra automorphism that has set-type square up to height~$1$ and induces $r_0$ by restriction.

\begin{proposition}\label{vectores constantes} If $r$ satisfies the braid equation, then
\begin{equation}
\begin{alignedat}{2}\label{simp}
&\alpha_l(a_i)(a_k,b_l)=\alpha_l(a_0)(a_0,b_0),&&\qquad \alpha_r(a_i)(a_k,b_l)=\alpha_r(a_0)(a_0,b_0),\\
&\alpha_l(b_j)(a_k,b_l)=\alpha_l(b_0)(a_0,b_0),&&\qquad \alpha_r(b_j)(a_k,b_l)=\alpha_r(b_0)(a_0,b_0),\\
&\beta_l(a_i)(a_k,b_l)=\beta_l(a_0)(a_0,b_0),&&\qquad \beta_r(a_i)(a_k,b_l)=\beta_r(a_0)(a_0,b_0),\\
&\beta_l(b_j)(a_k,b_l)=\beta_l(b_0)(a_0,b_0),&&\qquad \beta_r(b_j)(a_k,b_l)=\beta_r(b_0)(a_0,b_0),
\end{alignedat}
\end{equation}
for all $i$, $j$, $k$ and $l$.
\end{proposition}

\begin{proof} Since $\varphi=\phi_l\circ \phi_r$ acts as an $u$-cycle on $\{a_0,\dots, a_{u-1}\}$ and as an $v$-cycle on $\{b_0,\dots,b_{v-1}\}$ and $\gcd(u,v)=1$, this follows by Corollary~\ref{alfas y betas periodicidad} and equalities~\eqref{alpha constante} and~\eqref{beta constante}.
\end{proof}

In the sequel we assume that equalities~\eqref{simp} are fulfilled. For the sake of simplicity, from now on we write
\begin{align*}
&\alpha_{la}\coloneqq \alpha_l(a_0)(a_0,b_0),&& \alpha_{ra}\coloneqq \alpha_r(a_0)(a_0,b_0),\\
&\alpha_{lb}\coloneqq \alpha_l(b_0)(a_0,b_0), && \alpha_{rb}\coloneqq \alpha_r(b_0)(a_0,b_0),\\
&\beta_{la}\coloneqq \beta_l(a_0)(a_0,b_0),&& \beta_{ra}\coloneqq \beta_r(a_0)(a_0,b_0),\\
&\beta_{lb}\coloneqq \beta_l(b_0)(a_0,b_0),&& \beta_{rb}\coloneqq \beta_r(b_0)(a_0,b_0).
\end{align*}
By equalities~\eqref{simp} and Proposition~\ref{alpha y beta en quasi involutivo},
\begin{equation}\label{igualdad previa de alphas}
\alpha_{ra} = \alpha_{la}^{-1},\quad \alpha_{rb} = \alpha_{lb}^{-1},\quad \beta_{ra} = - \beta_{la}\alpha_{la}^{-1}\quad\text{and}\quad
\beta_{rb} = - \beta_{lb}\alpha_{lb}^{-1}.
\end{equation}
Combining this with equality~\eqref{split} we obtain that
\begin{equation}\label{igualdad de alphas}
\alpha_{la}^2 = \alpha_{lb}^2.
\end{equation}

\begin{remark}\label{simplificacion'} Let $C_r(a_i,b_j)$, $C_l(a_i,b_j)$ and $C_m(a_i,b_j)$ be as in Remark~\ref{condiciones para longitud menor que uno 1}. By equalities~\eqref{simp}, we know that $C_r(a_i,b_j)=C_l(a_i,b_j)=C_m(a_i,b_j)=1$ for all $i$ and~$j$.
\end{remark}

\begin{notation}\label{definicion gamma}
Set $\alpha\coloneqq\alpha_{la}$, $\beta_a\coloneqq \beta_{la}$, $\beta_b\coloneqq \beta_{lb}$ and $\Gamma_{i|j|k|l}\coloneqq \Gamma_{a_i|b_j|a_k|b_l}$.
\end{notation}

\begin{remark}\label{simplificacion} By equalities~\eqref{igualdad previa de alphas} and equality~\eqref{igualdad de alphas} there exists $\varepsilon\in\{\pm 1\}$ such that the following equalities hold:
\begin{equation}\label{definicion parametros}
\alpha_{lb}=\varepsilon\alpha,\quad \alpha_{ra}=\frac{1}{\alpha},\quad \alpha_{rb}=\frac{\varepsilon}{\alpha}, \quad \beta_{ra}=-\frac{\beta_a}{\alpha}\quad\text{and}\quad \beta_{rb}=-\frac{\varepsilon\beta_b}{\alpha}.
\end{equation}
\end{remark}

\begin{proposition}\label{proposition igualdad epsilon}
Let $\varepsilon$, $\alpha$, $\beta_a$ and $\beta_b$ be as above. Equality~\eqref{eq braided} holds for all $S\subseteq T$ such that $\mathfrak{h}(T)\le 1$ if and only if
\begin{equation}\label{igualdad para epsilon}
\beta_b(\alpha-1)=\beta_a(\varepsilon\alpha-1).
\end{equation}
\end{proposition}

\begin{proof} By \cite{GGV1}*{Proposition~4.5}, Proposition~\ref{complemento de la prop 4.5 de [10]} and the discussion in \cite{GGV1}*{Subsection~4.1} we know that equality~\eqref{eq braided} is satisfied for all $S\subseteq T$ such that $\mathfrak{h}(T)\le 1$ if and only if the conditions in \cite{GGV1}*{Proposition~4.5}, Proposition~\ref{complemento de la prop 4.5 de [10]} and item~5) of~\cite{GGV1}*{Subsection~4.1} are fulfilled.

Let $w$ be a fixed primitive $uv$-root of unity and let $\ell_j$, $\wp_j$, $\gamma_r$ and $\gamma_l$ be as in items~3) and~4) of~\cite{GGV1}*{Proposition~4.5}. By Remark~\ref{simplificacion'}, we can take $\gamma_l=\gamma_r=1$ and $\ell_j=\wp_j=1$.

Assume that equality~\eqref{eq braided} hold for all $S\subseteq T$ such that $\mathfrak{h}(T)\le 1$. By item~4) of~\cite{GGV1}*{Proposition~4.5} we know that, for all $i$,
\begin{equation}\label{vectores}
\biggl(\alpha_{la}-w^i,\sum_{j=0}^{uv-1} w^{ij}\beta_{la}\biggr)\sim \biggl(\alpha_{lb}-w^i,\sum_{j=0}^{uv-1}w^{ij}\beta_{lb}\biggr).
\end{equation}
Taking $i=0$ and using equalities~\eqref{definicion parametros}, we obtain
$$
\left(\alpha-1,uv\beta_a\right)\sim \left(\varepsilon \alpha - 1,uv\beta_b\right),
$$
which is equivalent to equality~\eqref{igualdad para epsilon}.

Conversely, assume that~\eqref{igualdad para epsilon} holds. By equalities~\eqref{simp}, in order to check that equality~\eqref{eq braided} is satisfied for all $S\subseteq T$ such that $\mathfrak{h}(T)\le 1$ it suffices to verify item~5) of \cite{GGV1}*{Subsection~4.1} and that, for each $0\le i<uv$,
\begin{align}
& \biggl(\alpha_{ra}-w^i,\sum_{j=0}^{v-1} w^{ij}\beta_{ra}\biggr)\sim \biggl(\alpha_{rb}-w^i,\sum_{j=0}^{v-1} w^{ij}\beta_{rb}\biggr),\label{alineados 1}\\
\shortintertext{and}
& \biggl(\alpha_{la}-w^i,\sum_{j=0}^{v-1}w^{ij}\beta_{la}\biggr)\sim \biggl(\alpha_{lb}-w^i,\sum_{j=0}^{v-1} w^{ij}\beta_{lb}\biggr).\label{alineados 2}
\end{align}
When $i\ne 0$, the second coordinate of all the vectors in~\eqref{alineados 1} and~\eqref{alineados 2} vanishes, so conditions~\eqref{alineados 1} and~\eqref{alineados 2} hold. When $i= 0$, the computation above shows that~\eqref{alineados 2} is equivalent to~\eqref{igualdad para epsilon}, and similarly~\eqref{alineados 1} is equivalent to~\eqref{igualdad para epsilon}. Finally, item~5) of~\cite{GGV1}*{Subsection~4.1} holds if and only if
$$
\beta_{li}+\alpha_{li}\beta_{rj}=\beta_{rj}+\alpha_{rj}\beta_{li}\quad\text{for all $i,j\in\{a,b\}$.}
$$
But these equalities follow from a straightforward computation using~\eqref{igualdad para epsilon}.
\end{proof}

\begin{remark}\label{clasificacion}
Since $\varepsilon=\pm 1$ and $\alpha\ne 0$, equality~\eqref{igualdad para epsilon} yields the following possible cases for the values of $\alpha$, $\beta_a$ and $\beta_b$:
\begin{enumerate}[itemsep=0.7ex, topsep=0.7ex, label=\arabic*)]

\item If $\varepsilon=1$, then $\alpha=1$ or $\beta_a=\beta_b$.

\item If $\varepsilon=-1$ and $\chr K\ne 2$, then either $\alpha=1$ and $\beta_a=0$, or $\alpha\ne 1$ and $\beta_b=\beta_a\frac{1+\alpha}{1-\alpha}$.
\end{enumerate}
\end{remark}

In the expression $\Gamma_{i|j|k|l}$ the indices $i$ and $k$ belong to $\{0,\dots,v-1\}$ and the indices~$j$ and~$l$ belong to $\{0,\dots,u-1\}$. So, the first ones can be affected by the maps $\sigma_a$ or $\tau_a$, while the second one, by the maps $\sigma_b$ or $\tau_b$. Since there is not danger of confusion, we will write $\Gamma_{\tau_i|\tau_j|\tau_k|\tau_l}$ instead of $\Gamma_{\tau_a(i)|\tau_b(j)|\tau_a(k)|\tau_b(l)}$, etcetera.

\begin{proposition}\label{pepe}
Let $\alpha$, $\beta_a$ and $\beta_b$ be as in Notation~\ref{definicion gamma}. Recall that by Remark~\ref{simplificacion} there exists $\varepsilon\in\{\pm 1\}$ such that equalities~\eqref{definicion parametros} are satisfied. Then equalities~\eqref{Caso 110}-- \eqref{Caso 111} hold if and only if for all $i,k,m\in \{0,\dots, u-1\}$ and $j,l,n\in \{0,\dots,v-1\}$ the following equalities are fulfilled:
\begin{align}
&\alpha^3\Gamma_{i|j|k|l}-\alpha\Gamma_{\tau_i|\tau_j|\tau_k|\tau_l}= \beta_a(1-\varepsilon\alpha)(\beta_a+\alpha\beta_b),\label{110 bajo}\\[3pt]
&\alpha^3\Gamma_{i|j|k|l} - \alpha \Gamma_{\tau_i|\tau_j|\tau_k|\tau_l}= \beta_b(1 - \alpha)(\beta_a + \alpha\beta_b),\label{110 alto}\\[3pt]
&\alpha^3 \Gamma_{i| j| \sigma_m| \sigma_n} - \alpha \Gamma_{\tau_i| \tau_j| m| n}= \beta_a(\alpha\beta_b(1-\varepsilon)+\beta_a(1-\alpha^2)),\label{101 bajo}\\[3pt]
&\alpha^3\Gamma_{i| j| \sigma_m| \sigma_n} - \alpha \Gamma_{\tau_i| \tau_j| m| n}= \beta_b(\alpha\beta_a(\varepsilon-1)+\beta_b(1-\alpha^2)),\label{101 alto}\\[3pt]
&\alpha^3\Gamma_{\sigma_k |\sigma_l|\sigma_m| \sigma_n} - \alpha\Gamma_{k|l|m|n}= \beta_a(\varepsilon\beta_b+\alpha \beta_a)(\varepsilon-\alpha),\label{011 bajo}\\[3pt]
&\alpha^3\Gamma_{\sigma_k |\sigma_l|\sigma_m| \sigma_n} - \alpha\Gamma_{k|l|m|n}=\beta_b(\beta_b+\varepsilon\alpha \beta_a)(1-\alpha),\label{011 alto}\\[3pt]
\begin{split}\label{111 bajo}
&\beta_a\beta_b(\beta_b\varepsilon(1+\alpha) -\beta_a(1+\varepsilon \alpha))= \alpha^2\beta_a(1+\alpha)\Gamma_{i|j|k|l}\\
&\phantom{\beta_a\beta_b} - \beta_b(1+\varepsilon \alpha)\Gamma_{\tau_i | \tau_j|\tau_k| \tau_l} + \alpha(\beta_b-\varepsilon\beta_a)\Gamma_{\tau_i| \tau_j|m|n}\\
&\phantom{\beta_a\beta_b} - \varepsilon\alpha^2(\beta_a-\beta_b)\Gamma_{i|j|\sigma_m | \sigma_n}
-\alpha^2\beta_b(1+\varepsilon\alpha)\Gamma_{\sigma_k |\sigma_l|\sigma_m |\sigma_n}\\
&\phantom{\beta_a\beta_b}  + \beta_a(1+\alpha)\Gamma_{k|l|m|n}.
\end{split}
\end{align}
\end{proposition}

\begin{proof} By~\eqref{split}, we have
\begin{align}
&\lambda_{a_i|b_j|a_k|b_l}^{{}^{(1)}\hf a_k|{}^{(1)}\hf a_k|a_i^{(1)}|b_j^{(1)}}=
\lambda_{a_i|b_j|a_k|a_k}^{{}^{(1)}\hf a_k|{}^{(1)}\hf a_k|a_i^{(1)}|b_j^{(1)}}
\lambda_{b_j|b_j|a_k|b_l}^{{}^{(1)}\hf a_k|{}^{(1)}\hf a_k|b_j^{(1)}|b_j^{(1)}}
=\alpha_{ra}\beta_{lb},\label{alfa por beta}\\
&\lambda_{a_i|b_j|a_k|b_l}^{{}^{(1)}\hf a_k|{}^{(1)}\hf b_l|a_i^{(1)}|a_i^{(1)}}=
\lambda_{a_i|a_i|a_k|b_l}^{{}^{(1)}\hf a_k|{}^{(1)}\hf b_l|a_i^{(1)}|a_i^{(1)}}
\lambda_{a_i|b_j|b_l|b_l}^{{}^{(1)}\hf b_l|{}^{(1)}\hf b_l|a_i^{(1)}|a_i^{(1)}}
=\alpha_{la}\beta_{rb}\label{alfa por beta 2}
\shortintertext{and}
&\lambda_{a_i|b_j|a_k|b_l}^{{}^{(1)}\hf a_k|{}^{(1)}\hf b_l|a_i^{(1)}|b_j^{(1)}}=
\lambda_{a_i|a_i|a_k|b_l}^{{}^{(1)}\hf a_k|{}^{(1)}\hf b_l|a_i^{(1)}|a_i^{(1)}}
\lambda_{a_i|b_j|b_l|b_l}^{{}^{(1)}\hf b_l|{}^{(1)}\hf b_l|a_i^{(1)}|b_j^{(1)}}
=\alpha_{la}\alpha_{rb}.\label{alfa por alfa}
\end{align}
A direct computation using these facts and equalities~\eqref{definicion parametros} shows that equality~\eqref{Caso 110} with $a\coloneqq a_i$, $b\coloneqq b_j$, $c\coloneqq a_k$, $d\coloneqq b_l$ and $e\coloneqq a_m$ becomes
$$
\Gamma_{i|j|k|l} + \frac{\varepsilon \beta_b \beta_a}{\alpha} - \frac{\beta_a\beta_b}{\alpha^2} + \frac{\varepsilon\beta_a^2}{\alpha^2} = \cancel{\frac{\beta_a^2}{\alpha^2}} - \cancel{\frac{\beta_a^2}{\alpha^2}} + \frac{\beta_a^2}{\alpha^3} + \frac{1}{\alpha^2}\Gamma_{\tau_i|\tau_j|\tau_k|\tau_l}.
$$
Multiplying by $\alpha^3$ and reordering we see that equality~\eqref{Caso 110} with these values of $a$, $b$, $c$, $d$ and $e$ is fulfilled if and only if
\begin{align*}
\alpha^3\Gamma_{i|j|k|l}-\alpha\Gamma_{\tau_i|\tau_j|\tau_k|\tau_l} &=\alpha\beta_a\beta_b-\varepsilon\alpha^2\beta_a\beta_b -\varepsilon\alpha\beta_a^2 +\beta_a^2\\
& = \beta_a(1-\varepsilon\alpha)(\beta_a+\alpha\beta_b),
\end{align*}
as desired. Similar arguments prove that equality~\eqref{Caso 110} with $a\coloneqq a_i$, $b\coloneqq b_j$, $c\coloneqq a_k$, $d\coloneqq b_l$ and $e\coloneqq b_n$ is fulfilled if and only if equality~\eqref{110 alto} holds, etcetera.
\end{proof}

For the rest of the section we assume that equality~\eqref{eq braided} is satisfied for all $S\subseteq T$, such that $\mathfrak{h}(T)\le 1$, or equivalently, that~\eqref{igualdad para epsilon} holds.

\smallskip

Let $\alpha$, $\beta_a$ and $\beta_b$ be as in Notation~\ref{definicion gamma} and let $\varepsilon\in\{\pm 1\}$ be as in Remark~\ref{simplificacion}. In order to abbreviate the expressions we write $\Gamma$ instead of $\Gamma_{0|0|0|0}$.

\begin{lemma}\label{simplificacion lemma} The right hand sides of equalities~\eqref{110 bajo}--\eqref{011 alto} are equal to
$$
\beta_a\beta_b(1-\varepsilon \alpha^2).
$$
\end{lemma}

\begin{proof}
Straightforward using~\eqref{igualdad para epsilon}.
\end{proof}

\begin{lemma}\label{casi todo}
Equalities~\eqref{110 bajo}--\eqref{011 alto} are fulfilled if and only if $\Gamma_{i|j|k|l}=\Gamma$ for all $i,j,k,l\frac{}{}$ and $(\alpha^2-1)\Gamma  = \frac{\beta_a\beta_b}{\alpha}(1-\varepsilon \alpha^2)$.
\end{lemma}

\begin{proof} Note that by Lemma~\ref{simplificacion lemma}, equalities~\eqref{110 bajo}, \eqref{101 bajo} and~\eqref{011 bajo} yield
\begin{align}
&\alpha^2\Gamma_{i|j|k|l}-\Gamma_{\tau_i|\tau_j|\tau_k|\tau_l}=\frac{\beta_a\beta_b}{\alpha}(1-\varepsilon\alpha^2),\label{110 simple}\\
&\alpha^2 \Gamma_{i| j| \sigma_k|\sigma_l} - \Gamma_{\tau_i| \tau_j| k| l}= \frac{\beta_a\beta_b}{\alpha}(1-\varepsilon\alpha^2)\label{101 simple}
\shortintertext{and}
&\alpha^2\Gamma_{\sigma_i |\sigma_j|\sigma_k| \sigma_l} - \Gamma_{i|j|k|l}=\frac{\beta_a\beta_b}{\alpha}(1-\varepsilon\alpha^2).\label{011 simple}
\end{align}

\smallskip

\noindent $\Rightarrow$)\enspace Recall that the map $(i,j)\mapsto (\varsigma_a(i), \varsigma_b(j))$ is an $uv$-cycle. From equality~\eqref{110 simple} we obtain that
$$
\alpha^2\Gamma_{\sigma_i|\sigma_j|\sigma_i|\sigma_j}- \Gamma_{\varsigma_i| \varsigma_j| \varsigma_i| \varsigma_j} =\frac{\beta_a\beta_b}{\alpha}(1-\varepsilon\alpha^2).
$$
Combining this with~\eqref{011 simple} we deduce that $\Gamma_{\varsigma_i| \varsigma_j| \varsigma_i| \varsigma_j} = \Gamma_{i|j|i|j}$ for all $i$ and $j$, which implies that
\begin{equation}\label{gamma constante en la diagonal}
\Gamma_{i|j|i|j}=\Gamma\qquad\text{for all $i$ and $j$.}
\end{equation}
On the other hand, from equality~\eqref{101 simple} we obtain that
$$
\alpha^2 \Gamma_{i| j| \varsigma_k| \varsigma_l} - \Gamma_{\tau_i| \tau_j| \tau_k| \tau_l} = \frac{\beta_a\beta_b}{\alpha}(1-\varepsilon\alpha^2).
$$
Combining this with~\eqref{110 simple}, we deduce that
$$
\Gamma_{i|j|k|l}=\Gamma_{i|j|\varsigma_k|\varsigma_l}\quad\text{for all $i$, $j$, $k$ and $l$.}
$$
From this and~\eqref{gamma constante en la diagonal} it follows that $\Gamma_{i|j|k|l}=\Gamma$ for all $i$, $j$, $k$ and $l$. Combining this with equality~\eqref{110 simple} we obtain that $(\alpha^2-1)\Gamma = \frac{\beta_a\beta_b}{\alpha}(1-\varepsilon \alpha^2)$, as desired.

\smallskip

\noindent $\Leftarrow$)\enspace By Lemma~\ref{simplificacion lemma}, this is clear.
\end{proof}

\begin{proposition}\label{Gammas contantes y equivalencia}
Equalities~\eqref{110 bajo}--\eqref{111 bajo} hold if and only if
\begin{align}
&\Gamma_{i|j|k|l}=\Gamma\quad\text{for all $i$, $j$, $k$ and $l$,}\label{los Gamas son iguales}\\
&(\alpha^2-1)\Gamma  = \frac{\beta_a\beta_b}{\alpha}(1-\varepsilon \alpha^2)\label{formulas equivalentes}\\
\shortintertext{and}
\begin{split}\label{formulas equivalentes 1}
& \beta_a\beta_b(\beta_b\varepsilon(1+\alpha)-\beta_a(1+\varepsilon \alpha))= \Gamma\bigl(\beta_a(1+\alpha)(1-\varepsilon \alpha+\alpha^2)\\
&\phantom{\beta_a\beta_b(\beta_b\varepsilon(1+\alpha)-\beta_a(1+\varepsilon \alpha))} -\beta_b(1+\varepsilon \alpha)(1-\alpha+\alpha^2)\bigr).
\end{split}
\end{align}
\end{proposition}

\begin{proof}
By Lemma~\ref{casi todo} it suffices to prove that if equality~\eqref{formulas equivalentes}
is fulfilled and $\Gamma_{i|j|k|l}=\Gamma$ for all $i$, $j$, $k$ and $l$, then equalities~\eqref{111 bajo} and~\eqref{formulas equivalentes 1} are equivalent, which follows by a direct computation.
\end{proof}

Recall that $X$ is the set $\{a_0,\dots,a_{u-1},b_0,\dots,b_{v-1}\}$, where $u,v\in\mathds{N}$ are coprime, endowed
with the height one order given by $a_i<b_j$ for all $i,j$. Recall also that $\phi_r,\phi_l\colon X\to X$ are two commuting poset automorphisms and
$$
r_0\colon X\times X\longrightarrow X\times X
$$
is the map defined by $r_0(x,y)\coloneqq (\phi_l(y),\phi_r(x))$. Let $\sigma_a$, $\sigma_b$, $\tau_a$ and $\tau_b$ be as at the beginning of the section. Recall finally that $(i,j)\mapsto (\varsigma_a(i),\varsigma_b(j))$ is an $uv$-cycle. Let
$$
r\colon D\otimes D\longrightarrow D\otimes D
$$
be a non-degenerate coalgebra automorphism that has set-type square up to height~$1$ and induces $r_0$ by restriction. By Theorem~\ref{EYB para longitud mayor que uno} and Propositions~\ref{proposition igualdad epsilon}, \ref{pepe} and~\ref{Gammas contantes y equivalencia} we know that $r$ satisfies the braid equation if and only it conditions~\eqref{igualdad para epsilon}, \eqref{los Gamas son iguales}, \eqref{formulas equivalentes} and~\eqref{formulas equivalentes 1} are fulfilled. In Remark~\ref{clasificacion} we described the solutions of equation~\eqref{igualdad para epsilon}. In the sequel we are going to construct the families of solutions of the braid equation corresponding to each of the cases of that remark. By condition~\eqref{los Gamas son iguales} we assume that $\Gamma_{i|j|k|l}=\Gamma$ for all $i$, $j$, $k$ and $l$.

\begin{proposition}\label{caso alpha uno y epsilon 1}
If $\varepsilon=1$ and $\alpha=1$, then $r$ satisfies the braid equation if and only if either $\beta_a=\beta_b$, or $\beta_a\ne \beta_b$ and $2\Gamma+2\beta_b\beta_a=0$.
\end{proposition}

\begin{proof}
A straightforward computation shows that under the hypotheses of the statement, conditions~\eqref{formulas equivalentes} and~\eqref{formulas equivalentes 1} are satisfied if and only if $\beta_a=\beta_b$, or $\beta_a\ne \beta_b$ and $2\Gamma+2\beta_b\beta_a=0$.
\end{proof}

\begin{proposition}
Assume that $\varepsilon=1$, $\alpha\ne 1$ and $\beta_a=\beta_b$.
\begin{enumerate}[itemsep=0.7ex, topsep=0.7ex, label=\arabic*)]

  \item If $\alpha=-1$, then $r$ satisfies the braid equation.

  \item If $\alpha\ne -1$, then $r$ satisfies the braid equation if and only if $\Gamma=-\frac{\beta_a^2}{\alpha}$.

\end{enumerate}

\end{proposition}

\begin{proof} Because under the conditions of item~1), equations~\eqref{formulas equivalentes} and~\eqref{formulas equivalentes 1} are always satisfied, while under the conditions of item~2), equations~\eqref{formulas equivalentes} and~\eqref{formulas equivalentes 1} are satisfied if and only if $\Gamma=-\beta_a^2/\alpha$.
\end{proof}

\begin{proposition}
Assume that $\varepsilon=-1$ and $\chr K\ne 2$.
\begin{enumerate}[itemsep=0.7ex, topsep=0.7ex, label=\arabic*)]

\item If $\alpha=1$ and $\beta_a=0$, then $r$ satisfies the braid equation.

\item If $\alpha=-1$ and $\beta_b=0$, then $r$ satisfies the braid equation.

\item If $\alpha^2\ne 1$ and $\beta_b=\beta_a\frac{1+\alpha}{1-\alpha}$, then $r$ satisfies the braid equation if and only if $\Gamma=-\frac{\beta_a^2(1+\alpha^2)}{\alpha(1-\alpha)^2}$.

\end{enumerate}
\end{proposition}

\begin{proof} Clearly, under the conditions of item~1), equations~\eqref{formulas equivalentes} and~\eqref{formulas equivalentes 1} are always satisfied. Assume we are under the hypotheses of item~3). Then Equality~\eqref{formulas equivalentes} holds if and only if
$$
\Gamma=-\frac{\beta_a\beta_b}{\alpha}\frac{1+\alpha^2}{1-\alpha^2}=-\frac{\beta_a^2(1+\alpha^2)}{\alpha(1-\alpha)^2}.
$$
Replacing $\Gamma$ by this, $\beta_b$ by $\beta_a\frac{1+\alpha}{1-\alpha}$ and $\varepsilon$ by $-1$ in both sides of~\eqref{formulas equivalentes 1},
we obtain
$$
-2\beta_a^3(1+\alpha^2)\frac{1+\alpha}{1-\alpha}=-2\beta_a^3(1+\alpha^2)\frac{1+\alpha}{1-\alpha}.
$$
Hence $r$ satisfies the braid equation if and only if $\Gamma=-\frac{\beta_a^2(1+\alpha^2)}{\alpha(1-\alpha)^2}$.
\end{proof}

For each $(a,b)$, $(c,d)$, $(e,f)$ and~$(g,h)$ in~$Y$, let~$\lambda_{a|b|c|d}^{e|f|g|h}$ be as in formula~\eqref{def de los lambdas}.

\begin{theorem}\label{teorema parametros} Let $r$ be as above of Proposition~\ref{caso alpha uno y epsilon 1}. If $r$ is a solution of the braid equation, then the possibly nonzero coefficients $\lambda_{a|b|c|d}^{e|f|g|h}$ depend on the parameters $\varepsilon$, $\alpha$, $\beta_a$, $\beta_b$ and $\Gamma$ via the following formulas, in which we use the notation~\eqref{potencia de phi}:
\begin{alignat*}{4}
&\lambda_{x|x|y|y}^{{}^{(1)}\hf y|{}^{(1)}\hf y|x^{(1)}|x^{(1)}}&&=1,&&&&\\[2pt]
&\lambda_{a_i|a_i|a_k|b_l}^{{}^{(1)}\hf a_k|{}^{(1)}\hf a_k|a_i^{(1)}|a_i^{(1)}}&&=\beta_a, &&\qquad\quad
\lambda_{b_j|b_j|a_k|b_l}^{{}^{(1)}\hf a_k|{}^{(1)}\hf a_k|b_j^{(1)}|b_j^{(1)}}&&=\beta_b,\\[2pt]
&\lambda_{a_i|b_j|a_k|a_k}^{{}^{(1)}\hf a_k|{}^{(1)}\hf a_k|a_i^{(1)}|a_i^{(1)}}&&=-\frac{\beta_a}{\alpha},&&\qquad\quad \lambda_{a_i|b_j|b_l|b_l}^{{}^{(1)}\hf b_l|{}^{(1)}\hf b_l|a_i^{(1)}|a_i^{(1)}}&&=-\frac{\varepsilon\beta_b}{\alpha},
 \\[2pt]
&\lambda_{a_i|a_i|a_k|b_l}^{{}^{(1)}\hf b_l|{}^{(1)}\hf b_l|a_i^{(1)}|a_i^{(1)}}&&=-\beta_a, &&\qquad\quad
\lambda_{b_j|b_j|a_k|b_l}^{{}^{(1)}\hf b_l|{}^{(1)}\hf b_l|b_j^{(1)}|b_j^{(1)}}&&=-\beta_b,\\[2pt]
& \lambda_{a_i|b_j|a_k|a_k}^{{}^{(1)}\hf a_k|{}^{(1)}\hf a_k|b_j^{(1)}|b_j^{(1)}}&&=\frac{\beta_a}{\alpha}, &&\qquad\quad \lambda_{a_i|b_j|b_l|b_l}^{{}^{(1)}\hf b_l|{}^{(1)}\hf b_l|b_j^{(1)}|b_j^{(1)}}&&=\frac{\varepsilon\beta_b}{\alpha}, \\[2pt]
 &\lambda_{a_i|a_i|a_k|b_l}^{{}^{(1)}\hf a_k|{}^{(1)}\hf b_l|a_i^{(1)}|a_i^{(1)}}&&=\alpha, &&\qquad\quad
 \lambda_{b_j|b_j|a_k|b_l}^{{}^{(1)}\hf a_k|{}^{(1)}\hf b_l|b_j^{(1)}|b_j^{(1)}}&&=\varepsilon  \alpha,\\[2pt]
&\lambda_{a_i|b_j|a_k|a_k}^{{}^{(1)}\hf a_k|{}^{(1)}\hf a_k|a_i^{(1)}|b_j^{(1)}}&&=\frac{1}{\alpha},&&\qquad\quad
\lambda_{a_i|b_j|b_l|b_l}^{{}^{(1)}\hf b_l|{}^{(1)}\hf b_l|a_i^{(1)}|b_j^{(1)}}&&=\frac{\varepsilon}{\alpha},\\[2pt]
&\lambda_{a_i|b_j|a_k|b_l}^{{}^{(1)}\hf a_i|{}^{(1)}\hf a_i|a_k^{(1)}|a_k^{(1)}}&&=\Gamma, &&\qquad\quad
\lambda_{a_i|b_j|a_k|b_l}^{{}^{(1)}\hf b_l|{}^{(1)}\hf b_l|a_i^{(1)}|a_i^{(1)}}&&=\frac{\varepsilon \beta_a\beta_b}{\alpha},\\[2pt]
& \lambda_{a_i|b_j|a_k|b_l}^{{}^{(1)}\hf a_k|{}^{(1)}\hf a_k|b_j^{(1)}|b_j^{(1)}}&&=\frac{ \beta_a\beta_b}{\alpha}, &&\qquad\quad
\lambda_{a_i|b_j|a_k|b_l}^{{}^{(1)}\hf b_l|{}^{(1)}\hf b_l|b_j^{(1)}|b_j^{(1)}}&&=-\Gamma-\frac{\beta_a\beta_b}{\alpha}(1+\varepsilon),
\\[2pt]
&\lambda_{a_i|b_j|a_k|b_l}^{{}^{(1)}\hf a_k|{}^{(1)}\hf b_l|a_i^{(1)}|a_i^{(1)}}&&=-\varepsilon\beta_b, &&\qquad\quad
\lambda_{a_i|b_j|a_k|b_l}^{{}^{(1)}\hf a_k|{}^{(1)}\hf a_k|a_i^{(1)}|b_j^{(1)}}&&=\frac{\beta_b}{\alpha},\\[2pt]
& \lambda_{a_i|b_j|a_k|b_l}^{{}^{(1)}\hf b_l|{}^{(1)}\hf b_l|a_i^{(1)}|b_j^{(1)}}&&=-\frac{\varepsilon \beta_a}{\alpha}, &&\qquad\quad
 \lambda_{a_i|b_j|a_k|b_l}^{{}^{(1)}\hf a_k|{}^{(1)}\hf b_l|b_j^{(1)}|b_j^{(1)}}&&=\varepsilon\beta_a, \\[2pt]
& \lambda_{a_i|b_j|a_k|b_l}^{{}^{(1)}\hf a_k|{}^{(1)}\hf b_l|a_i^{(1)}|b_j^{(1)}}&&=\varepsilon. &&\qquad\quad
\end{alignat*}
\end{theorem}

\begin{proof}
The first equality holds because $r$ induces $r_0$. The other formulas can be obtained using~\eqref{conjuntista}, \eqref{split}, the discussion at the beginning of Section~\ref{seccion A family of examples}, Proposition~\ref{vectores constantes} and equalities~\eqref{definicion parametros}. We prove for example that
\begin{equation}\label{ejemplo}
\lambda_{a_i|b_j|a_k|b_l}^{{}^{(1)}\hf b_l|{}^{(1)}\hf b_l|a_i^{(1)}|b_j^{(1)}}=-\frac{\varepsilon \beta_a}{\alpha}.
\end{equation}
By~\eqref{split} and the discussion at the beginning of Section~\ref{seccion A family of examples} we know that
$$
\lambda_{a_i|b_j|a_k|b_l}^{{}^{(1)}\hf b_l|{}^{(1)}\hf b_l|a_i^{(1)}|b_j^{(1)}}= \lambda_{a_i|a_i|a_k|b_l}^{{}^{(1)}\hf b_l|{}^{(1)}\hf
b_l|a_i^{(1)}|a_i^{(1)}} \lambda_{a_i|b_j|b_l|b_l}^{{}^{(1)}\hf b_l|{}^{(1)}\hf b_l|a_i^{(1)}|b_j^{(1)}}
=-\beta_{la}\alpha_{rb}.
$$
We finish the computation of the equality~\eqref{ejemplo}, noting that by Proposition~\ref{vectores constantes} and equalities~\eqref{definicion
parametros} we have $\beta_{la}=\beta_a$ and $\alpha_{rb}^i=\frac{\varepsilon}{\alpha}$.
\end{proof}

Collecting the results in this section we arrive at the following complete description:

\begin{theorem}\label{teorema principal}
The non-degenerate coalgebra automorphism $r$ introduced above Proposition~\ref{caso alpha uno y epsilon 1} is a solution of the braid equation that induces $r_0$ on $X\times X$ and has set-type square up to height~$1$ if and only if the parameters $\varepsilon$, $\alpha$, $\beta_a$, $\beta_b$
and $\Gamma$ belong to one of the families given in the following table:
\end{theorem}
\setlength{\tabcolsep}{4pt}
\ra{1.2}
\setfracpadding
\begin{longtabu}{lllll}
\caption{Families for $\varepsilon$, $\alpha$, $\beta_a$, $\beta_b$ and $\Gamma$} \\
\toprule
$\#$ &  \parbox[c]{.209\linewidth} {\begin{center} Fixed values in each family\end{center}}
&  \parbox[c]{.209\linewidth} {\begin{center} Dependent values in each family\end{center}}
& \parbox[c]{.209\linewidth} {\begin{center} Parameters \end{center}}
& \parbox[c]{.14\linewidth} {\begin{center} $\chr K$ \end{center}}\\
\midrule
\endfirsthead
$\#$ &  \parbox[c]{.209\linewidth} {\begin{center} Fixed values in each family\end{center}}
&  \parbox[c]{.209\linewidth} {\begin{center} Dependent values in each family\end{center}}
& \parbox[c]{.209\linewidth} {\begin{center} Parameters \end{center}} \\
\midrule
\endhead
\endfoot
\bottomrule
\endlastfoot
1. & $\varepsilon=1$, $\alpha=1$& $\beta_b=\beta_a$ & $\beta_a,\Gamma\in K$ &arbitrary\\
\midrule
2.& $\varepsilon=1$, $\alpha=1$&  $\Gamma=-\beta_a\beta_b$ & $\beta_a,\beta_b\in K$ &$\chr K \ne 2$\\
   &  &  & $\beta_b\ne \beta_a$\\
\midrule
3. & $\varepsilon=1$, $\alpha=1$& & $\beta_a,\beta_b,\Gamma\in K$ &$\chr K=2$\\
   &  &  & $\beta_b\ne \beta_a$\\
\midrule
4. & $\varepsilon=1$, $\alpha=-1$& $\beta_b=\beta_a$ & $\beta_a,\Gamma\in K$&arbitrary \\
\midrule
5.& $\varepsilon=1$ & $\beta_b=\beta_a$ & $\beta_a\in K$, $\alpha\in K^{\times}$&arbitrary \\
   &  &$\Gamma=-\frac{\beta_a^2}{\alpha}$    & $\alpha^2\ne 1$\\
\midrule
6. & $\varepsilon=-1$, $\alpha=1$, $\beta_a=0$ && $\beta_b,\Gamma\in K$ &arbitrary\\
\midrule
7. & $\varepsilon=-1$, $\alpha=-1$, $\beta_b=0$ && $\beta_a,\Gamma\in K$ &arbitrary\\
\midrule
8. & $\varepsilon=-1$, &$\beta_b=\beta_a\frac{1+\alpha}{1-\alpha}$ & $\beta_a$, $\alpha\in K^{\times}$ &arbitrary\\
&  &$\Gamma=-\frac{\beta_a^2(1+\alpha^2)}{\alpha(1-\alpha)^2}$  &$\alpha^2\ne 1$
\label{tabla}
\end{longtabu}

\begin{proposition}\label{cuando es de cuadrado conjuntista}
The solutions~$r$ of Theorem~\ref{teorema principal} have set-type square if and only if either $\varepsilon=-1$ and $2\Gamma = 0$ or $\varepsilon=1$ and $2\Gamma=-\frac{2\beta_a\beta_b}{\alpha}$.
\end{proposition}

\begin{proof} Let $r$ be as Theorem~\ref{teorema principal}. By Proposition~\ref{vectores constantes} and equality~\eqref{los Gamas son iguales} all the equations~\eqref{para altura 1}
reduce to
$$
2\Gamma= - \frac{\beta_b\beta_a}{\alpha} - \frac{\beta_b\beta_a}{\varepsilon\alpha},
$$
which is clearly equivalent to $\varepsilon=-1$ and $2\Gamma = 0$ or $\varepsilon=1$ and $2\Gamma=-\frac{2\beta_a\beta_b}{\alpha}$. By Proposition~\ref{rel a prec b y c prec d} this finishes the proof.
\end{proof}

\begin{examples}
The hypothesis of Theorems~\ref{teorema parametros} and~\ref{teorema principal} are satisfied for instance in the following cases
\begin{itemize}[itemsep=0.7ex, topsep=0.7ex]
  \item[-] $u=4$, $v=1$, $\phi_r(a_i)=a_{i+1}$ and $\phi_l(a_i)=a_{i+2}$, where the sums are taken modulo $4$.

  \item[-] $u$ and $v$ are odd, $\phi_r(a_i)=\phi_l(a_i)=a_{i+1}$ and $\phi_r(b_i)=\phi_l(b_i)=b_{i+1}$, where the sums are taken modulo $u$ and $v$ respectively.

   \item[-] Any example can be constructed in the following way: take an $u$-cycle $\varsigma_a$ of~$\mathds{N}_u$ and a $v$-cycle $\varsigma_b$ of~$\mathds{N}_v$, and then take $\sigma_a$ and $\tau_a$ any powers of $\varsigma_a$ such that $\sigma_a\circ \tau_a=\varsigma_a$, and take $\sigma_b$ and $\tau_b$ any powers of $\varsigma_b$ such that $\sigma_b\circ \tau_b=\varsigma_b$.  In fact, if $\varsigma_a=\sigma_a\circ \tau_a=\tau_a\circ\sigma_a$ is an $u$-cycle, then $\sigma_a$ and $\tau_a$ are in the centralizer of $\varsigma_a$ in $S_u$ which is $\langle \varsigma_a\rangle$. The same argument proves that both $\sigma_b$ and $\tau_b$ are powers of $\varsigma_b$. For instance, one can take $\sigma_a=\varsigma_a$, $\sigma_b=\varsigma_b$, $\tau_a=\ide$ and $\tau_b=\ide$.
\end{itemize}
\end{examples}

\end{document}